\documentclass[a4paper,cleveref, autoref,numberwithinsect]{amsart}
\usepackage{graphicx}
\usepackage{color}
\usepackage{amsmath}
\usepackage{amssymb}
\usepackage{xspace}
\usepackage{epsfig}
\usepackage[dvipsnames]{xcolor}
\usepackage{soul}
\usepackage{xfrac}
\usepackage{enumitem}
\usepackage{hyperref}

\usepackage{algorithm}
\usepackage{algorithmicx}
\usepackage{algpseudocode}
\newcommand {\mm}[1] {\ifmmode{#1}\else{\mbox{\(#1\)}}\fi}

\newcommand{\ignore}[1]{}

\newsavebox{\smallProofsym}                 
\savebox{\smallProofsym}                               %
{
\begin{picture}(6,6)
\put(0,0){\framebox(6,6){}}
\put(0,2){\framebox(4,4){}}
\end{picture} 
}  

\makeatletter
\long\def\@makecaption#1#2{%
  \vskip\abovecaptionskip
  \sbox\@tempboxa{\small #1: #2}%
  \ifdim \wd\@tempboxa >\hsize
    \small #1: #2\par
  \else
    \global \@minipagefalse
    \hb@xt@\hsize{\hfil\box\@tempboxa\hfil}%
  \fi
  \vskip\belowcaptionskip}
\makeatother

\theoremstyle{plain}
\newtheorem*{namedthm}{\namedthmname}
\newcounter{namedthm}



\theoremstyle{plain}
\newtheorem{theorem}{Theorem}[section]

\newtheorem{lemma}[theorem]{Lemma}

\theoremstyle{definition}
\newtheorem{definition}{Definition}

\theoremstyle{plain}

\newtheorem*{supermaintheorem*}{Main Theorem}

\newtheorem*{supermaincorollary*}{Main Corollary}

\newcommand{\Hgroup}[2]     {\mm{{\sf H}_{#1}{({#2})}}}

\newcommand{\Rspace}        {\mm{{\mathbb R}}}
\newcommand{\Zspace}        {\mm{{\mathbb Z}}}

\newcommand{\Depth}[2]      {\mm{{\rm Depth}_{#1}{({#2})}}}
\newcommand{\Succ}[3]       {\mm{\rm Succ}_{#1}^{#2}{({#3})}}
\newcommand{\Pred}[3]       {\mm{\rm Pred}_{#1}^{#2}{({#3})}}

\newcommand{\BR}            {\mm{{B}}}
\newcommand{\DR}            {\mm{{D}}}
\newcommand{\BD}[2]         {\mm{{{\it BD}}_{#1}{({#2})}}}
\newcommand{\Dgm}[2]        {\mm{{\rm Dgm}_{#1}{({#2})}}}
\newcommand{\dime}[1]       {\mm{\rm dim\,}{#1}}

\newcommand{\Skip}[1]       {}

\title{The Depth Poset under Transpositions in the Filter}
\author[Edelsbrunner]{Herbert Edelsbrunner}
\address{ISTA (Institute of Science and Technology Austria), Kloster\-neu\-burg, Austria}
\email{herbert.edelsbrunner@ist.ac.at}
\author[Lipi\'{n}ski]{Micha{\l} Lipi\'{n}ski}
\address{ISTA (Institute of Science and Technology Austria), Kloster\-neu\-burg, Austria}
\email{michal.lipinski@ist.ac.at}
\author[Mrozek]{Marian Mrozek}
\address{Division of Computational Mathematics, Faculty of Mathematics and Computer \newline\indent Science, Jagiellonian University, Krak\'ow, Poland}
\email{marian.mrozek@uj.edu.pl}
\author[Soriano-Trigueros]{Manuel Soriano-Trigueros}
\address{
	\parbox{\linewidth}{Universidad de Sevilla, Seville, Spain\\
		ISTA (Institute of Science and Technology Austria), Kloster\-neu\-burg, Austria}
	}
\email{manuel.sorianotrigueros@ist.ac.at}
\author[Zimin]{Fedor Zimin}
\address{ISTA (Institute of Science and Technology Austria), Kloster\-neu\-burg, Austria}
\email{fedor.zimin@ist.ac.at}
\keywords{Algebraic topology, Lefschetz complexes, persistent homology, vines and vineyards, birth-death pairs, shallow pairs, relations, partial orders, transpositions.}

\thanks
{\footnotesize
The first author is partially supported by the DFG Collaborative Research Center TRR 109, Austrian Science Fund (FWF), grant no.\ {I 02979-N35.}
The second author acknowledges that this project has received funding from the European Union’s Horizon 2020 research and innovation programme under the Marie Skło\-dow\-ska-Curie Grant Agreement No.~101034413.
The third author is partially supported by Polish National Science Center under Opus Grant 2019/35/B/ST1/00874.
}

\def\corcommstyle#1{\bf\small\tt\textcolor{red}{#1}}
\def\corrc #1<<>>{
  \begin{quote} {\corcommstyle{ $<<$COMMENT$>>$ #1\marginpar{!!}}} \end{quote}
}

\begin{document}
\maketitle

\begin{abstract}
  The depth poset of a filtered Lefschetz complex reflects the dependencies between the cancellations of different shallow birth-death pairs.
  Using the fast algorithms for computing the depth poset in \cite{ELMS25} and for updating the persistence diagram under transpositions in \cite{CEM06}, we give a complete case analysis of how transpositions of cells in the filter affect the depth poset.
  In addition, we present statistics on the depth poset for random point data and its sensitivity to the transpositions that occur in random straight-line homotopies.
\end{abstract}

%%%%%%%%%%%%%%%%%%%%%%%%%%%%%%%%%%%%%%%%%%%%%%%
%%%%%%%%%%%%%%%%%%%%%%%%%%%%%%%%%%%%%%%%%%%%%%%
\section{Introduction}
\label{sec:1}
%%%%%%%%%%%%%%%%%%%%%%%%%%%%%%%%%%%%%%%%%%%%%%%
%%%%%%%%%%%%%%%%%%%%%%%%%%%%%%%%%%%%%%%%%%%%%%%

The field of \emph{topology optimization} aims at the design of shapes whose large- and small-scale connectivity is subject to modifications applied in order to satisfy constraints and optimize objective functions.
The traditional approach is numerical and optimizes through repeated local improvement; see the text by Bends{\o}e~\cite{Ben95} and the major revision by Bends{\o}e and Sigmund~\cite{BeSi04}.
More recently, the methods of topological data analysis have matured and provide an alternative approach to this seasoned problem; see e.g.\ the ``big steps'' algorithm by Nigmetov and Morozov~\cite{NiMo24}, which is closely related to the work presented in this paper.
Specifically, we aim at exploring the space of shapes that relate to each other through the cancellation and interchange of the critical points of a function.
This is naturally related to the simplification of functions and shapes, as studied in \cite{AGHLM09,BLW11,EMP06}, but goes beyond this work by casting light on the space of possible simplifications.

\smallskip
A second motivation for the presented work is the development of a full-fledged discrete theory of combinatorial dynamics.
Being rooted in the discrete approach to Morse theory pioneered by Forman~\cite{For98,For98b}, the overarching goal is to provide discrete counterparts for classic continuous concepts, such as the Conley index and the Morse decomposition; see \cite{LKMW23,Mro17} for important steps in this direction.

\smallskip
The primary technical prerequisites for our work are the depth poset of a filter on a complex, as introduced in \cite{ELMS25}, the well established concept of persistent homology; see the text by Edelsbrunner and Harer~\cite{EdHa10}, and the vineyard algorithm for computing persistence along a parametrized family of functions, as originally described in \cite{CEM06}.
The depth poset collects all dependencies between cancellations of \emph{shallow birth-death pairs} (formerly known as \emph{apparent pairs} \cite{Bau21} and \emph{close pairs} \cite{DRS15}) and thus provides a global yet discrete representation of the space of shapes reachable by sequences of such cancellations.
We mention that shallow pairs appeared in other seemingly unrelated contexts, for example in fast adaptive sorting \cite{RSWa24}.
Our main contribution is the complete case analysis of how transpositions in the filter affect the depth poset.
Working toward establishing the depth poset as a first class object, we also present statistics on the size and shape of this poset and its variation under transpositions.

\medskip \noindent \textbf{Outline.}
Section~\ref{sec:2} explains the technical background needed to describe our results.
Section~\ref{sec:3} gives the complete case analysis of how a transposition in the filter affects the depth poset of the same.
Section~\ref{sec:4} presents statistical results that probe the depth poset and homotopies between them for random PL height functions on Delaunay mosaics of Poisson point processes.

%% \newpage
%%%%%%%%%%%%%%%%%%%%%%%%%%%%%%%%%%%%%%%%%%%%%%%
%%%%%%%%%%%%%%%%%%%%%%%%%%%%%%%%%%%%%%%%%%%%%%%
\section{Background}
\label{sec:2}
%%%%%%%%%%%%%%%%%%%%%%%%%%%%%%%%%%%%%%%%%%%%%%%
%%%%%%%%%%%%%%%%%%%%%%%%%%%%%%%%%%%%%%%%%%%%%%%

This section presents the background we need to study the relations and orders in which our results are coached.
Besides standard material, we need concepts from persistent homology and its matrix reduction algorithms.
We refer to \cite{EdHa10} for more comprehensive background.

\medskip \noindent \textbf{Relations.}
We adapt the terminology from graph theory, in which a relation consists of \emph{(directed) arcs} that connect \emph{nodes}.
Each arc goes from its \emph{source} to its \emph{target}, both of which are nodes.
A \emph{(directed) path} is a sequence of arcs such that the target of each arc is the source of the next arc, and it goes from the source of the first arc to the target of the last arc.
It is a \emph{directed cycle} if the target of the last arc is also the source of the first arc.
The relation is \emph{acyclic} if it has no directed cycle.
Note however that an acyclic relation may have an \emph{undirected cycle}; that is: two disjoint paths that connect the same two nodes.

\smallskip
An acyclic relation has a unique \emph{transitive closure}, which is the possibly larger (acyclic) relation that has an arc from a source to a target whenever there is such a path.
A \emph{partial order} is an acyclic relation that is its own transitive closure, and in situations where we wish to emphasize the nodes rather than the arcs, we refer to it as a \emph{partially ordered set}, or a \emph{poset}, for short.
Two nodes in a poset are \emph{comparable} if there is an arc that connects them, and we write $a \leq b$ if $a$ is the source and $b$ is the target of the connecting arc.
A \emph{total order} is a partial order in which any two nodes are comparable.
It is a \emph{linear extension} of a partial order on the same nodes if $a \leq b$ in the partial order implies $a \leq b$ in the total order.
\begin{definition}
  \label{dfn:face_relation}
  The nodes of the \emph{face relation} of a complex are its cells, with an arc from a cell, $a$, to another cell, $b$, if $a$ is a face of $b$ and $\dime{a} = \dime{b} - 1$.
  We call $a$ a \emph{facet} of $b$ and $b$ a \emph{cofacet} of $a$.
  A \emph{Lefschetz complex} avoids the complications that come with concrete geometric descriptions of the cells by treating them as abstract entities.
  It lists the facets of a cell but does not specify how they are attached to the cell.
\end{definition}
We refer to \cite{ELMS25} for a more formal definition of a Lefschetz complex, a concept which was introduced first under the more modest name of a \emph{complex} in the book by Lefschetz~\cite{Lef42}.
A function on a complex, $f \colon K \to \Rspace$, that satisfies $f(a) < f(b)$ whenever $a \prec b$ is called a \emph{filter} of $K$.
The ordering in which $a$ precedes $b$ iff $f(a) < f(b)$ is then a linear extension of the face relation; see the left panel of Figure~\ref{fig:Example}.
\begin{figure}[t]
  \centering \vspace{0.05in}
  \resizebox{!}{1.8in}{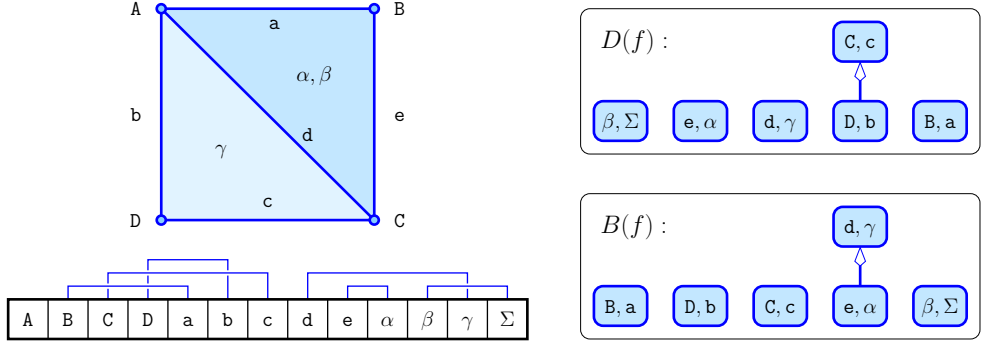}
  \vspace{-0.05in}
  \caption{{\footnotesize \emph{Upper left:} a complex with four vertices, five edges, three triangles ($\alpha, \beta, \gamma$, in which $\alpha$ and $\beta$ have the same boundary), and one $3$-cell ($\Sigma$, which is sandwiched between $\alpha$ and $\beta$).
  \emph{Lower left:} the ordering implied by a filter, with birth-death pairs as indicated (see definition below).
  \emph{Right:} the death and birth relations on the birth-death pairs as computed by Algorithms~1 and 2 below.}
	\vspace{-0.5cm}
}
  \label{fig:Example}
\end{figure}

\medskip \noindent \textbf{Persistent Homology and Transpositions.}
Given a filter on a complex, $f \colon K \to \Rspace$, we construct the complex by adding one cell at a time following the implied order.
After adding $i$ cells, we arrive at a subcomplex $K_i \subseteq K$, and we refer to the resulting sequence of subcomplexes as a \emph{filtration} of $K$.
Fixing a dimension, $p$, and a coefficient field, e.g.\ $\Zspace / 2 \Zspace$, we apply the degree-$p$ homology functor to each subcomplex and thus get a linear sequence of vector spaces connected by linear maps:
\begin{align}
  \Hgroup{p}{\emptyset} = \Hgroup{p}{K_0} \to \Hgroup{p}{K_1} \to \ldots \to \Hgroup{p}{K_n} = \Hgroup{p}{K} ;
\end{align}
see e.g.\ \cite{Hat02} for the needed background in algebraic topology.
Whenever the added cell has dimension $p$, it either gives \emph{birth} to a degree-$p$ homology class or it gives \emph{death} to a degree-$(p-1)$ homology class.
This is classic Morse theory, and the novelty in persistence is that for each death-giving $(p+1)$-cell we find a unique birth-giving $p$-cell whose class it ends.
We call these the degree-$p$ \emph{birth-death pairs} of $f$, denoted $\BD{p}{f}$.
It is also possible that some birth-giving $p$-cells do not get paired, and they reflect the degree-$p$ homology of $K$.
A convenient combinatorial representation of this information is the \emph{persistence diagram}, denoted $\Dgm{p}{f}$, which maps $(a,b) \in \BD{p}{f}$ to the point $(f(a), f(b)) \in \Rspace^2$.

\smallskip
An important property is the stability of the persistence diagram, originally proved in \cite{CEH07}.
In a nutshell, it says that for two monotonic functions, $f, g \colon K \to \Rspace$, there is a bijection between $\Dgm{p}{f}$ and $\Dgm{p}{g}$ such that the maximum coordinate difference between corresponding points is bounded from above by the absolute maximum of $f(a)-g(a)$ over all cells $a \in K$.
This will be instrumental in the following construction.

\smallskip
Consider two filters, $f_0 , f_1 \colon K \to \Rspace$, and the straight-line homotopy defined by $f_\lambda (a) = (1-\lambda) f_0(a) + \lambda f_1(a)$ for $0 \leq \lambda \leq 1$.
These functions imply a $1$-parameter family of persistence diagrams for each dimension $p$.
Mapping every point $(f_\lambda(a), f_\lambda(b)) \in \Dgm{p}{f_\lambda}$ to $(f_\lambda(a), f_\lambda(b), \lambda) \in \Rspace^3$, we effectively stack up these diagrams, and because of stability, the points form continuous curves in $\Rspace^3$.
Following the terminology in \cite{CEM06}, we call these curves \emph{vines} and the collection of vines a \emph{vineyard}.
Assuming $K$ is finite, there are only finitely many values of $\lambda$ for which $f_\lambda$ is not a filter (because it assigns the same value to at least two cells).
These values split $[0,1]$ into finitely many open intervals with a unique ordering of the cells in each.
In the generic case, the orderings in any two consecutive intervals differ by a single \emph{transposition} of two adjacent cells.
Such a transposition may or may not affect the birth-death pairs, and if it does, we call it a \emph{switch}.
As established in \cite{CEM06}, there are only three kinds of switches: the \emph{BB-type}, which swaps the birth-giving cells between two pairs, the \emph{DD-type}, which swaps the death-giving cells between two pairs, and the \emph{BD-type}, which swaps a birth-giving with a death-giving cell; see Figure~\ref{fig:switches}.
Observe that in the first two cases, the two affected points belong to the same-degree persistence diagram, while in the third case the affected points belong to persistence diagrams of consecutive degrees.
\begin{figure}[bt]
  \centering \vspace{0.0in}
  \resizebox{!}{1.3in}{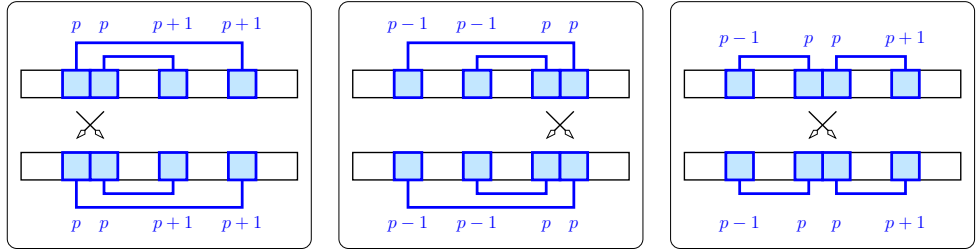}
  \vspace{-0.05in}
  \caption{{\footnotesize The three types of switches.
  \emph{From left to right:} swapping two birth-giving cells (BB-type), swapping two death-giving cells (DD-type), and swapping a birth-giving with a death-giving cell (BD-type).
  In the first two cases, the intervals that correspond to the pairs are nested---before and after the switch---and in the third case, they are disjoint---again before and after the switch.}
	\vspace{-0.3cm}
}
  \label{fig:switches}
\end{figure}

\medskip \noindent \textbf{Cancellation of Shallow Pairs.}
Given a filtered Lefschetz complex, we are interested in simplifying it.
The natural operation to this end is the \emph{cancellation} of a cell with one of its facets.
Letting $a \prec b$ be such a pair, the operation removes $a$ and $b$ and adds an arc from every facet of $b$ to every cofacet of $a$ in the face relation; see Figure~\ref{fig:cancellation}.
In the process, a cell may become the facet of another cell a multiple of times, which for $\Zspace / 2 \Zspace$ coefficients means it is not a facet if this multiple is even.
\begin{figure}[hbt]
  \centering \vspace{0.05in}
  \resizebox{!}{1.0in}{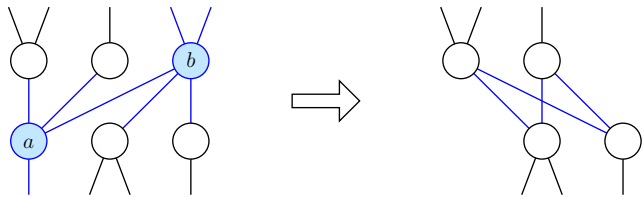}
  \vspace{-0.05in}
  \caption{{\footnotesize To cancel $a \prec b$, we connect all facets of $b$ to all cofacets of $a$.
  The remainder of the face relation is unchanged.}
	\vspace{-0.2cm}
}
  \label{fig:cancellation}
\end{figure}

A cancellation may have side-effects, such as making a cell a facet of another cell that precedes it in the filter.
To avoid the need to reorganize the filter, we restrict ourselves to cancellations where this is not necessary.
Even more, we restrict ourselves to canceling \emph{shallow pairs}, $a \prec b$, for which $a$ is the last facet of $b$ and $b$ is the first cofacet of $a$ in the ordering of the filter; see \cite{ELMS25} but also \cite{Bau21}, where they are referred to as \emph{apparent pairs}.
If $a \prec b$ is a shallow pair, then all other facets of $b$ precede $a$ and all other cofacets of $a$ succeed $b$, so the new filter is obtained by removing $a$ and $b$, and no reordering is necessary.
The following basic properties are either obvious or proved in \cite{ELMS25}:
\begin{itemize}
  \item every shallow pair is a birth-death pair of the filter;
  \item after canceling a shallow pair, the birth-death pairs of the new filter are the remaining birth-death pairs of the old filter;
  \item there is at least one shallow pair, unless the face relation is empty.
\end{itemize}
Because of the third property, we can keep canceling shallow pairs until the face relation runs empty.
We call the resulting sequence of canceled pairs a \emph{shallow cancellation order}, which by the first two properties is a total order of the birth-death pairs of the original filter.
However, there may be many orderings of the birth-death pairs such that each pair is shallow at the time it gets canceled, which motivates the next concept.
\begin{definition}
  \label{dfn:depth_poset}
  The \emph{depth poset} of a filter on a Lefschetz complex, $f \colon K \to \Rspace$, is the intersection of all shallow cancellation orders of $f$, denoted $\Depth{}{f}$.
\end{definition}
This concept was introduced in \cite{ELMS25}, along with some of its fundamental properties, which we state without proof.
\begin{lemma}
  \label{lem:properties_of_depth_poset}
  Let $f \colon K \to \Rspace$ be a filter on a Lefschetz complexs, $\varphi =(x,y)$ and $\psi = (a,b)$ two birth-death pairs of $f$, $(\varphi, \psi)$ an arc in $\Depth{}{f}$, $\sigma = (s,t)$ a shallow pair of $f$, and $f' \colon K \setminus \{s,t\}$ the filter after canceling $\sigma$.
  Then
  \begin{enumerate}
    \item nodes connected by an arc in $\Depth{}{f}$ are nested; that is: $f(a) < f(x) < f(y) < f(b)$;
    \item arcs of $\Depth{}{f}$ respect dimensions; that is: $\dime{a} = \dime{x}$ and $\dime{b} = \dime{y}$;
    \item canceling $\sigma$ has no side-effects; that is: $\Depth{}{f'} = \Depth{}{f} \cap [\BD{}{f'} \times \BD{}{f'}]$.
  \end{enumerate}
\end{lemma}
By Property~2, $\Depth{}{f}$ is the disjoint union of posets $\Depth{p}{f} \subseteq \BD{p}{f} \times \BD{p}{f}$.

\medskip \noindent \textbf{Matrix Reduction Twice.}
While the definition of the depth poset is constructive, it is in terms of possibly exponentially many total orders on the birth-death pairs.
Alternatively, it can be constructed as the transitive closure of two relations, both of which can be computed efficiently by reducing the boundary matrix of the complex.
We review the algorithms and refer to \cite{ELMS25} for a proof of correctness.
\begin{definition}
  \label{dfn:death_relation}
  The \emph{death relation} of a filter on a Lefschetz complex, $f \colon K \to \Rspace$, denoted $\DR (f) \subseteq \BD{}{f} \times \BD{}{f}$, records the left-to-right column operations during the reduction of the boundary matrix in reverse order of the birth-giving cells; see Algorithm~1.
\end{definition}
\begin{algorithm}[hbt]
  \caption{Bottom to Top Column Reduction}\label{alg:column_reduction}
  \begin{algorithmic}[1]
	\State{$R_1 = \Delta$; $U_1 = {\rm Id}$;}
	\While {$R_1 \neq 0$}
            \State{let $R_1[x,y]$ be the leftmost non-zero entry in the lowest non-zero row;}
		\While{$\exists b > y$ such that $R_1[x,b] = 1$}
            \State{add column $y$ to column $b$ in $R_1$;
            $U_1[y,b] = U_1[y,b]+1$}
            \EndWhile;
        \State{delete row $x$ and column $y$ from $R_1$}
	\EndWhile.
  \end{algorithmic}
\end{algorithm}
The arcs in the death relation correspond to the off-diagonal non-zero entries in the columns of death-giving cells in $U_1$:
\begin{align}
  \DR(f) &= \{ \left( (x,y) , (a,b) \right) \in \BD{}{f} \times \BD{}{f} \mid y \neq b {\rm ~and~}U_1[y,b] = 1\} .
\end{align}
To rationalize $U_1$, consider the matrix $V_1$ that stores the chains whose boundaries are the cycles in $R_1$.
Initially, $V_1$ is the identity matrix and it is maintained by undergoing the same column operations as $R_1$.
If we ignore the deletions of rows and columns, we retain the matrices in their original sizes, and it is not difficult to see that they satisfy $R_1 = \Delta V_1$ and $\Delta = R_1 U_1$ throughout the algorithm, so $U_1 = V_1^{-1}$.
\begin{definition}
  \label{dfn:birth_relation}
  The \emph{birth relation} of a filter on a Lefschetz complex, $f \colon K \to \Rspace$, denoted $\BR (f) \subseteq \BD{}{f} \times \BD{}{f}$, records the bottom-to-top row operations during the reduction of the boundary matrix in the order of the death-giving cells; see Algorithm~2.
\end{definition}
\begin{algorithm}[hbt]
  \caption{Left to Right Row Reduction}\label{alg:row_reduction}
  \begin{algorithmic}[1]
	\State{$R_2 = \Delta$; $U_2 = {\rm Id}$;}
	\While {$R_2 \neq 0$}
            \State{let $R_2[x,y]$ be the lowest non-zero entry in the leftmost non-zero column;}
		\While{$\exists a < x$ such that $R_2[a,y] = 1$}
            \State{add row $x$ to row $a$ in $R_2$; 
            $U_2[a,x] = U_2[a,x] + 1$}
            \EndWhile;
        \State{delete row $x$ and column $y$ from $R_2$}
	\EndWhile.
  \end{algorithmic}
\end{algorithm}
The arcs in the birth relation correspond to the off-diagonal non-zero entries in the rows of birth-giving cells in $U_2$:
\begin{align}
  \BR(f) &= \{ \left( (x,y), (a,b) \right) \in \BD{}{f} \times \BD{}{f} \mid x \neq a {\rm ~and~} U_2[a,x] = 1 \}.
\end{align}
To rationalize $U_2$, we introduce the matrix $V_2$ that stores the chains whose coboundaries are the cocycles, i.e.\ $R_2 = V_2 \Delta$ assuming we ignore the deletions of rows and columns.
Similarly $\Delta = U_2 R_2$, so $U_2 = V_2^{-1}$.

\smallskip
The right panel of Figure~\ref{fig:Example} shows examples of the death and birth relations of a filter.
A noteworthy feature of this example are the pairs $({\tt d}, \gamma)$ and $(\beta, \Sigma)$.
In spite of the fact that {\tt d} is an edge of $\beta$, which prevents $({\tt d}, \gamma)$ from being shallow, there is no arc connecting the two pairs in the depth poset.
Indeed, the two pairs are neither nested nor do their dimensions match, so an arc would contradict the first two properties in Lemma~\ref{lem:properties_of_depth_poset}.
Hence, it must be possible to cancel $({\tt d}, \gamma)$ before canceling $(\beta, \Sigma)$, and this can indeed be done if we first cancel $({\tt e}, \alpha)$, which is the sole predecessor of $({\tt d}, \gamma)$ in the poset.
Geometrically, the cancellation of $({\tt e},\alpha)$ attaches $\beta$ to {\tt d} twice, which in modulo-$2$ arithmetic is counted as no attachment.

\smallskip
It is not difficult to see that the two relations are neither necessarily their own transitive closures nor necessarily their own transitive reductions, and while they express different pieces of information, the are not necessarily disjoint.
Nevertheless, they determine the depth poset as the transitive closure of their union; see \cite{ELMS25} for a proof.
This is important so we state the claim more formally.
\begin{theorem}
 \label{thm:relations_combine_to_depth_poset}
 Let $f \colon K \to \Rspace$ be a filter on a Lefschetz complex.
 Then $\Depth{}{f}$ is the transitive closure of $\DR(f) \cup \BR(f)$.
\end{theorem}

%%%%%%%%%%%%%%%%%%%%%%%%%%%%%%%%%%%%%%%%%%%%%%%
%%%%%%%%%%%%%%%%%%%%%%%%%%%%%%%%%%%%%%%%%%%%%%%
\section{Case Analysis}
\label{sec:3}
%%%%%%%%%%%%%%%%%%%%%%%%%%%%%%%%%%%%%%%%%%%%%%%
%%%%%%%%%%%%%%%%%%%%%%%%%%%%%%%%%%%%%%%%%%%%%%%

A transposition preserves all birth-death pairs, unless it is a switch, in which case the two transposed cells replace each other in their respective pairs.
Consider for example the filter in Figure~\ref{fig:Example} and entertain what happens when we transpose the edges $\tt c$ and $\tt d$.
Comparing Figures~\ref{fig:Example} and \ref{fig:Example-2}, we note that $({\tt C}, {\tt c})$ and $({\tt d}, \gamma)$ got replaced by $({\tt C}, {\tt d})$ and $({\tt c}, \gamma)$, and the death and birth relations both lost their only arc.
\begin{figure}[hbt]
  \centering \vspace{0.05in}
  \resizebox{!}{1.70in}{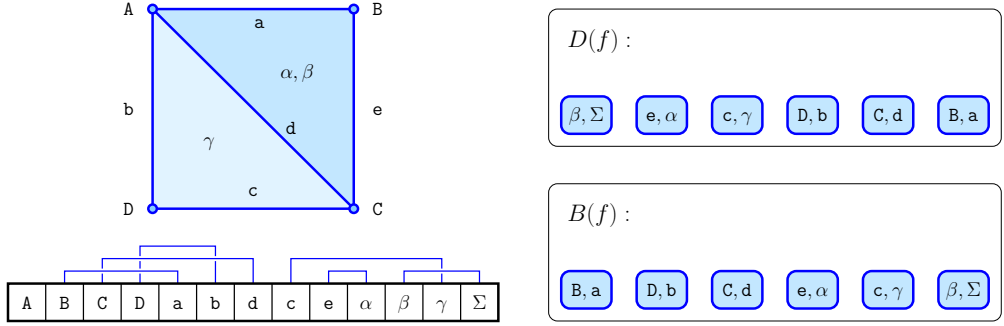}
  \vspace{-0.05in}
  \caption{{\footnotesize The only difference to the example in Figure~\ref{fig:Example} is the order of the edges $\tt c$ and $\tt d$.
  To the \emph{right}, we see the death and birth relations, which are both empty.}}
  \label{fig:Example-2}
\end{figure}
It should however be noted that a transposition may change the depth poset without affecting the birth-death pairs.
To see this, we organize a complete case analysis along the classification of transpositions into three types: of two birth-giving cells, of two death-giving cells, and of a birth-giving with a death-giving cell.

\smallskip
Prior to the case analysis, we state two lemmas that will both be repeatedly used.
The first excludes certain configurations of birth-death pairs.
\begin{lemma}
  \label{lem:no_zig}
  Let $f \colon K \to \Rspace$ be a filter on a Lefschetz complex, and $a \prec y$ two cells in $K$.
  If $a$ is the last facet of $y$ in the ordering implied by $f$, and $a$ is not yet paired when we encounter $y$ in the ordering from the left, then $(a,y)$ is a birth-death pair of $f$.
\end{lemma}
\begin{proof}
  Consider the standard matrix reduction algorithm for persistent homology, which reduces the columns of the ordered boundary matrix from left to right.
  Since $a$ is the last facet of $y$, the lowest non-zero entry in column $y$ is in row $a$.
  When the algorithm arrives at $y$, $a$ is not yet paired, so there is no pivot in its row.
  It follows that this entry remains and becomes a pivot, which is equivalent to $(a,y)$ becoming a birth-death pair.
\end{proof}

The second lemma addresses a property of ordered boundary matrices that amounts to a reformulation of Lemma~4.6 in \cite{ELMS25}.
We state it without proof.
\begin{lemma}
  \label{lem:independent_of_order}
  Let $f \colon K \to \Rspace$ be a filter on a Lefschetz complex, $\Delta$ the corresponding ordered boundary matrix, $x$ a row of $\Delta$, and $y$ a column of $\Delta$.
  Then any two boundary matrices obtained by canceling all birth-death pairs with pivots below row $x$ or to the left of column $y$ following shallow but possibly different cancellation orders are the same.
\end{lemma}

%%%%%%%%%%%%%%%%%%%%%%%%%%%%%%%%%%%%%%%%%%%%%%%
\subsection{Case I: Birth-birth Transpositions}
\label{sec:3.1}
%%%%%%%%%%%%%%%%%%%%%%%%%%%%%%%%%%%%%%%%%%%%%%%

Let $(x,y)$ and $(a,b)$ be two nested or otherwise overlapping birth-death pairs with matching dimension, so $f(a) < f(x) < f(y) < f(b)$ or $f(x) < f(a) < f(y) < f(b)$ and $\dime{a} = \dime{x} = \dime{y}-1 = \dime{b}-1$.
As illustrated in Figure~\ref{fig:CaseBB}, we assume that $a$ and $x$ are consecutive in the ordering, so either the two pairs are incomparable in the depth poset, or they are connected by an arc but not by a path of two arcs.
To prepare the analysis, we cancel the birth-death pairs whose pivots are below row $x$ or to the left of column $y$ in the ordered boundary matrix.
By Lemma~\ref{lem:independent_of_order}, the sequence in which these pairs are canceled does not affect the remaining portion of the boundary matrix.
After the cancellations, $(x,y)$ is shallow and $(a,b)$ is either shallow or has $(x,y)$ as its sole predecessor in the depth poset.
Similarly, we cancel all birth-death pairs with pivots below row $a$ or to the left of column $y$ in the boundary matrix after transposing $a$ and $x$.
The two operations commute, so we can alternatively transpose $a$ and $x$ after the mentioned cancellations.
Refer to Figure \ref{fig:CaseBB} for the three sub-cases as defined after the initial cancellations:
\begin{figure}[h!]
  \centering \vspace{0.05in}
  \resizebox{!}{5.2in}{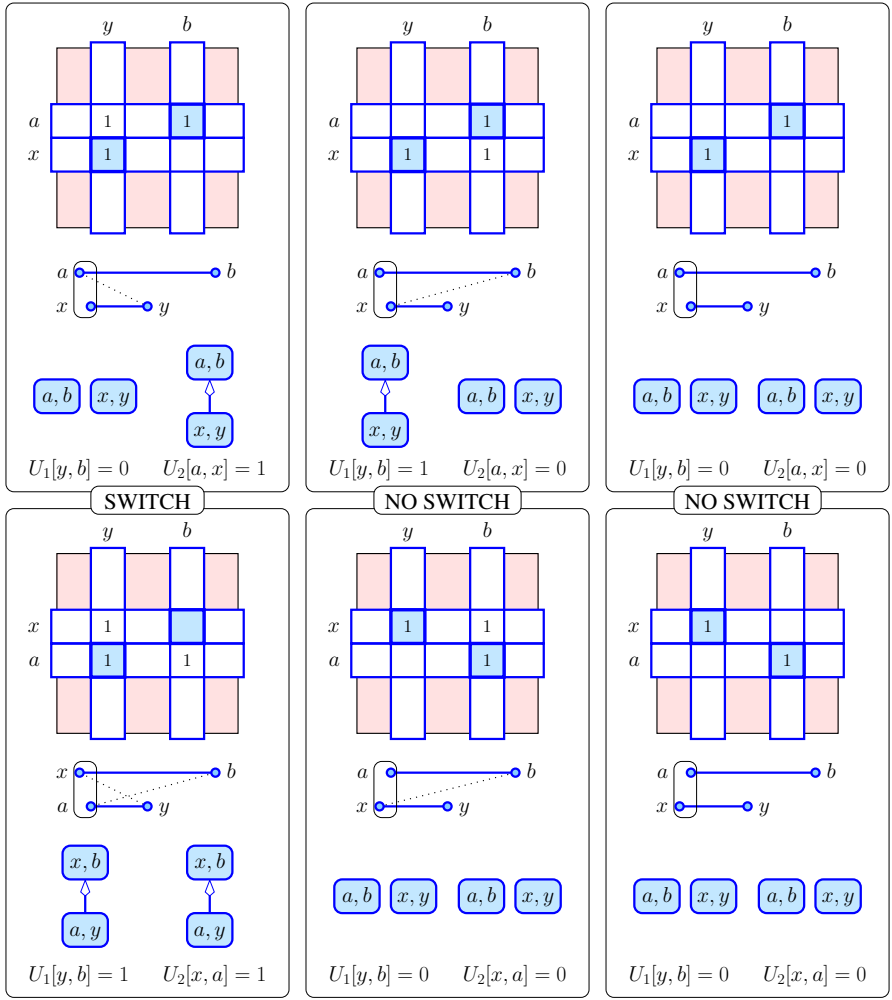}
  \vspace{-0.05in}
  \caption{{\footnotesize The configurations before and after the transposition of two birth-giving cells.
  \emph{From left to right:} Cases~I.1, I.2, I.3, with the configuration before and after the transposition at the \emph{top} and the \emph{bottom}, respectively, or vice versa for the inverse transposition.
  Within each panel, we show the relevant birth- and death-giving cells as rows and columns of the boundary matrix, respectively, the two birth-death pairs as bars with \emph{dotted} face relation, and the corresponding nodes and arcs in the death and birth relations.}}
  \label{fig:CaseBB}
\end{figure}
\begin{description}
  \item[Case I.1:] $a \prec y$ and $x \nprec b$, as illustrated in the upper left panel of Figure~\ref{fig:CaseBB}.
  Since $(a,b)$ is not shallow, it must be that $(x,y)$ precedes $(a,b)$ in the depth poset before the transposition of $a$ and $x$.
  To see that the transposition is a switch note that after transposing, $a$ is the last facet of $y$ and is not yet paired before we encounter $y$ from the left, so $(a,y)$ is a birth-death pair by Lemma~\ref{lem:no_zig}.
  By the stability of the persistence diagram \cite{CEH07}, $(x,b)$ is another pair after the transposition, and the remaining pairs are unchanged; see the lower left panel of Figure~\ref{fig:CaseBB}.
  Finally, reading the two panels from bottom to top, we have the inverse switch defined by $x \prec y$ and $a \prec b$.
  \item[Case I.2:] $a \nprec y$ and $x \prec b$, as illustrated in the upper middle panel of Figure~\ref{fig:CaseBB}.
  Like in Case~I.1, $(x,y)$ is a predecessor of $(a,b)$ before the transposition of $a$ and $x$.
  Since $a$ is not a facet of $y$, $(x,y)$ remains shallow, so the transposition is not a switch.
  But since the two pairs are no longer nested, they are also no longer comparable by Lemma~\ref{lem:properties_of_depth_poset}; see the lower middle panel of Figure~\ref{fig:CaseBB}.
  Reading the two panels from bottom to top, we have the inverse transposition defined by $a \nprec y$ and $x \prec b$, but now for two overlapping pairs.
  \item[Case I.3:] $a \nprec y$ and $x \nprec b$, as illustrated in the upper right panel of Figure~\ref{fig:CaseBB}.
  Like in Case~I.2, the transposition of $a$ and $x$ maintains $(x,y)$ as a shallow pair, so the transposition is not a switch.
  The two pairs are incomparable before the transposition because $x$ is not a facet of $b$ and $y$ is not a cofacet of $a$, and after the transposition because the pairs are no longer nested; see the lower right panel of Figure~\ref{fig:CaseBB}.
  The inverse transposition is defined by $a \nprec y$ and $x \nprec b$, but now for pairs that are neither nested nor disjoint.
\end{description}
When the transposition of $a$ and $x$ is not a switch, then the changes described in Cases~I.2 and I.3 are exhaustive, but if it is a switch, there may be changes beyond those described in Case~I.1.
We need notation to describe them.
Letting $\varphi = (x,y)$ and $\psi = (a,b)$ be two birth-death pairs, the former is a predecessor of the latter in the death relation if $U_1[y,b] = 1$ and in the birth relation if $U_2[a,x] = 1$.
We therefore write
\begin{align}
  \Pred{1}{}{a,b} &= \{ \varphi \mid (\varphi, \psi) \in \DR \} = \{ (x,y) \in \BD{}{f} \mid y \neq b {\rm ~and~} U_1[y,b] = 1 \}; 
    \label{eqn:pred1} \\
  \Pred{2}{}{a,b} &= \{ \varphi \mid (\varphi, \psi) \in \BR \} =\{ (x,y) \in \BD{}{f} \mid a \neq x {\rm ~and~} U_2[a,x] = 1 \}; 
    \label{eqn:pred2} \\
  \Succ{1}{}{a,b} &= \{ \varphi \mid (\psi, \varphi) \in \DR \} =\{ (x,y) \in \BD{}{f} \mid b \neq y {\rm ~and~} U_1[b,y] = 1 \}; 
    \label{eqn:succ1} \\
  \Succ{2}{}{a,b} &= \{ \varphi \mid (\psi, \varphi) \in \BR \} =\{ (x,y) \in \BD{}{f} \mid x \neq a {\rm ~and~}U_2[x,a] = 1 \},
    \label{eqn:succ2}
\end{align}
and $\Pred{1}{{\rm bef}}{a,b}$, $\Pred{1}{{\rm aft}}{a,b}$ for the sets before and after a transposition, etc.
We begin with the changes of the successors in the death relation and write $\oplus$ for the symmetric difference operation between sets.
\begin{lemma}
  \label{lem:changes_after_BB-type_switch}
  Let $f \colon K \to \Rspace$ be a filter on a Lefschetz complex, $(x,y)$ and $(a,b)$ birth-death pairs of matching dimension, $a$ precedes $x$, $y$ precedes $b$, and $a, x$ are consecutive in the ordering implied by the filter.
  Then $a \prec y$ iff $(x,y) \in \Pred{2}{}{a,b}$ iff the transposition of $x$ and $a$ is a BB-type switch (see the left panels in Figure~\ref{fig:CaseBB}), and in this case
{\small
  \begin{align}
	\Succ{1}{{\rm aft}}{a,y} &= \Succ{1}{{\rm bef}}{x,y} \oplus (x,b) \oplus \Succ{1}{{\rm bef}}{a,b}; & 
		\Succ{1}{{\rm aft}}{x,b} &= \Succ{1}{{\rm bef}}{a,b}; 
	\label{eqn:Lemma33a} \\
	\Succ{2}{{\rm aft}}{a,y} &= \Succ{2}{{\rm bef}}{x,y} \oplus \{ (x,b), (a,b) \}; &
		\Succ{2}{{\rm aft}}{x,b} &= \Succ{2}{{\rm bef}}{a,b}; 
	\label{eqn:Lemma33b} \\
	\Pred{1}{{\rm aft}}{a,y} &= \Pred{1}{{\rm bef}}{x,y};&
		\Pred{1}{{\rm aft}}{x,b} &= \Pred{1}{{\rm bef}}{a,b}; 
	\label{eqn:Lemma33c} \\
	\Pred{2}{{\rm aft}}{a,y} &= \Pred{2}{{\rm bef}}{x,y}; &
		\Pred{2}{{\rm aft}}{x,b} &= \Pred{2}{{\rm bef}}{a,b}.
      \label{eqn:Lemma33d}
  \end{align}
}%
\normalsize
  Otherwise, the depth relation remains as is, except if $x \prec b$ (see the middle panels in Figure~\ref{fig:CaseBB}), in which case
  \begin{align}
    \Succ{1}{{\rm aft}}{x,y} &= \Succ{1}{{\rm bef}}{x,y} \oplus (a,b) \oplus \Succ{1}{{\rm bef}}{a,b},
      \label{eqn:Lemma33e}
  \end{align}
  while all other sets of successors and predecessors of $(x,y)$ and $(a,b)$ remain unchanged.
\end{lemma}
\begin{proof}
  The two equivalences follow from the completeness of the case analysis displayed in Figure~\ref{fig:CaseBB}.
  To see \eqref{eqn:Lemma33a}, consider rows $a$ and $x$ in the boundary matrix before and after the switch.
  With reference to the upper left panel of Figure~\ref{fig:CaseBB}---which shows the matrix before the switch---we give names to the portions of these rows delimited by columns $y$ and $b$:
  \begin{align}
    M_a &= \{R_1[a,t] \mid f(y) \leq f(t) < f(b)\}; ~~ N_a = \{R_1[a,t] \mid f(b) \leq f(t)\};  \\
    M_x &= \{R_1[x,t] \mid f(y) \leq f(t) < f(b)\}; ~~ N_x = \{R_1[x,t] \mid f(b) \leq f(t)\}.
  \end{align}
  Canceling $(x,y)$ changes row $a$ to $(M_a + M_x) | (N_a + N_x)$, in which the bar means catenation.
  At this time, $(a,b)$ is a shallow pair, so $M_a + M_x = 0$, which implies $M_a = M_x$.
  Columns $y$ and $b$ in $U_1$ are transposed copies of rows $x$ and $a$ at the moment before $(x,y)$ and then $(a,b)$ are canceled.
  Hence, column $y$ in $U_1$ is $(M_x | N_x)^T$, and column $b$ is $(N_a+N_x)^T$.
  This is before the transposition of $a$ and $x$.
  After the transposition, column $y$ in $U_1$ is $(M_a | N_a)^T$, and column $b$ is $(N_x + N_a)^T$.
  Using $M_a = M_x$ and $N_a = N_x + (N_a + N_x)$, we rewrite these relations in different notation:
  \begin{align}
    U_1^{{\rm aft}} [y,t] &= \left\{ \begin{array}{ll}
        U_1^{{\rm bef}} [y,t]  &  {\rm for~~} f(y) \leq f(t) < f(b) , \\
        U_1^{{\rm bef}} [y,t] + U_1^{{\rm bef}} [b,t] &  {\rm for~~} f(b) \leq f(t) ; 
      \end{array} \right. 
      \label{eqn:U1fory} \\
    U_1^{{\rm aft}} [b,t] &= U_1^{{\rm bef}} [b,t] .
      \label{eqn:U1forb}
  \end{align}
  Finally, we translate the relations into the notation of the lemma using \eqref{eqn:succ1}, and get \eqref{eqn:Lemma33a} from \eqref{eqn:U1fory} and \eqref{eqn:U1forb}.
  We omit the arguments for \eqref{eqn:Lemma33b}, \eqref{eqn:Lemma33c}, \eqref{eqn:Lemma33d}, and \eqref{eqn:Lemma33e}.
\end{proof}
The symmetric cases---when $x$ precedes $a$ so the birth-birth transpositions go from the lower to the upper panes in Figure~\ref{fig:CaseBB}---are similar and the same relations between the successors and predecessors before and after the transposition apply.

%%%%%%%%%%%%%%%%%%%%%%%%%%%%%%%%%%%%%%%%%%%%%%%
\subsection{Case II: Death-death Transpositions}
\label{sec:3.2}
%%%%%%%%%%%%%%%%%%%%%%%%%%%%%%%%%%%%%%%%%%%%%%%

The setting is similar to Case~I: $(x,y)$ and $(a,b)$ are nested or otherwise overlapping birth-death pairs of matching dimension, but now we assume that $y$ and $b$ are consecutive in the ordering implied by the filter.
As before, we prepare the analysis by canceling all birth-death pairs whose pivots are below row $x$ or to the left of column $y$.
After these cancellations, $(x,y)$ is shallow, and $(a,b)$ is shallow or it has $(x,y)$ as its sole predecessor in the depth poset.
Refer to Figure~\ref{fig:CaseDD} for the three sub-cases as defined after the initial cancellations:
\begin{figure}[h!]
  \centering \vspace{0.05in}
  \resizebox{!}{5.2in}{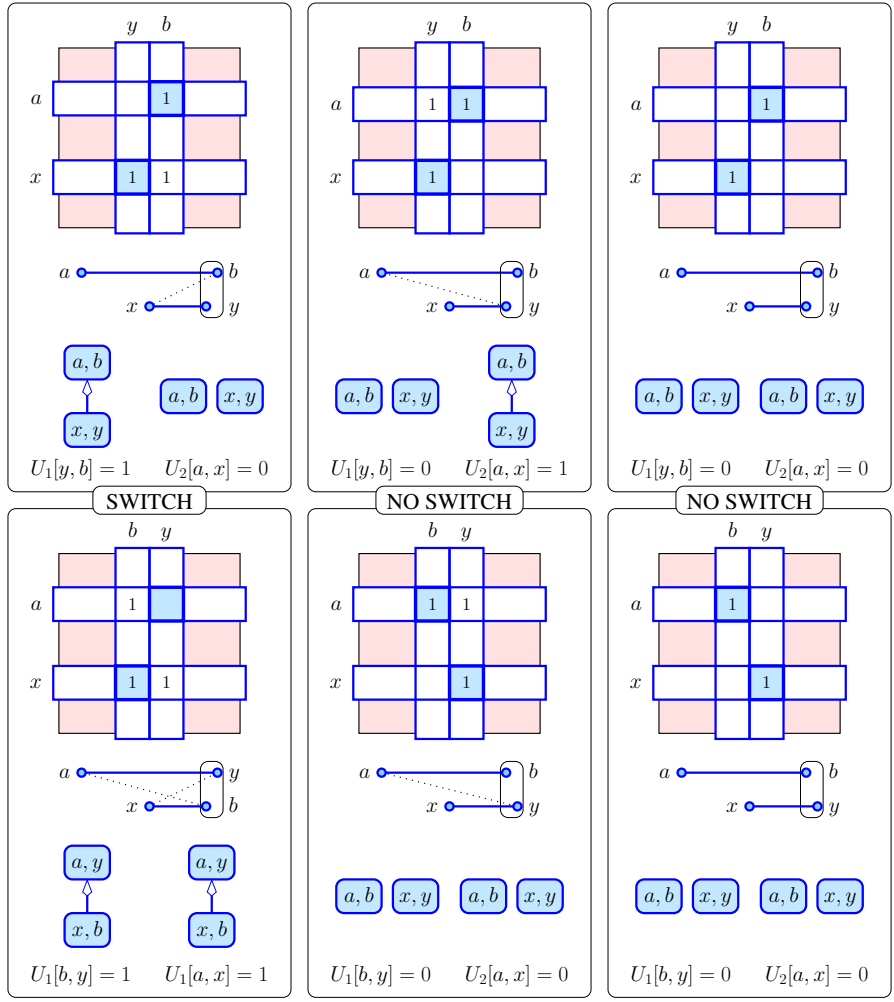}
  \vspace{-0.05in}
  \caption{{\footnotesize The configurations before and after the transposition of two death-giving cells in the filter.
  \emph{From left to right:} Cases~II.1, II.2, II.3, with the configuration before the transposition at the \emph{top} and after the transposition at the \emph{bottom}, or vice versa for the inverse transposition.
  Within each panel, we show the relevant birth-giving cells as rows and death-giving cells as columns of the matrix, the two birth-death pairs as bars, and the corresponding nodes and relations in the death and birth relations.}}
  \label{fig:CaseDD}
\end{figure}
\begin{description}
  \item[Case II.1:] $a \nprec y$ and $x \prec b$ for a nested pair, as illustrated in the upper left panel of Figure~\ref{fig:CaseDD}.
  The transposition is a switch, and reading the two left panels from bottom to top, we have the inverse switch defined by $x \prec y$ and $a \prec b$.
  \item[Case II.2:] $a \prec y$ and $x \nprec b$ for a nested pair, as illustrated in the upper middle panel of Figure~\ref{fig:CaseDD}.
  The transposition is not a switch, and reading the two middle panels from bottom to top, we have the inverse transposition defined by the same two conditions for overlapping pairs.
  \item[Case II.3:] $a \nprec y$ and $x \nprec b$ for a nested pair, as illustrated in the upper right panel of Figure~\ref{fig:CaseDD}.
  Again this transposition is not a switch, and reading the two right panels from bottom to top, we have the inverse transposition defined by the same two conditions for overlapping pairs.
\end{description}
The sub-cases are strictly symmetric to Cases~I.1, I.2, I.3, with identical arguments for the symmetric actions.
While omitting other details, we state the result symmetric to Lemma~\ref{lem:changes_after_BB-type_switch}, which describes the changes to the death and birth relations beyond the pairs that contain the transposed cells.
\begin{lemma}
  \label{lem:changes_after_DD-type_switch}
  Let $f \colon K \to \Rspace$ be a filter on a Lefschetz complex, $(x,y)$, $(a,b)$ birth-death pairs of matching dimension, $a$ precedes $x$, $y$ precedes $b$, and $y, b$ are consecutive in the ordering implied by the filter.
  Then $x \prec b$ iff $(x,y) \in \Pred{1}{}{a,b}$ iff the transposition of $y$ and $b$ is a DD-type switch (see the left panels in Figure~\ref{fig:CaseDD}). and in this case
\small
\begin{align}
	\Succ{1}{{\rm aft}}{x,b} &= \Succ{1}{{\rm bef}}{x,y} \oplus \{ ((a,y), (a,b) \};&	
		\Succ{1}{{\rm aft}}{a,y} &= \Succ{1}{{\rm bef}}{a,b}; 
	\label{eqn:Lemma34a} \\
	\Succ{2}{{\rm aft}}{x,b} &= \Succ{2}{{\rm bef}}{x,y} \oplus (a,y) \oplus \Succ{2}{{\rm bef}}{a,b};& 
		\Succ{2}{{\rm aft}}{a,y} &= \Succ{2}{{\rm bef}}{a,b};
	\label{eqn:Lemma34b} \\
	\Pred{1}{{\rm aft}}{x,b} &= \Pred{1}{{\rm bef}}{x,y};& 
		\Pred{1}{{\rm aft}}{a,y} &= \Pred{1}{{\rm bef}}{a,b}; 
	\label{eqn:Lemma34c} \\
	\Pred{2}{{\rm aft}}{x,b} &= \Pred{2}{{\rm bef}}{x,y};&
		\Pred{2}{{\rm aft}}{a,y} &= \Pred{2}{{\rm bef}}{a,b}. 
	\label{eqn:Lemma34d} 
\end{align}
\normalsize
  Otherwise, the depth relation remains as is, except if $a \prec y$ (see the middle panels in Figure~\ref{fig:CaseDD}), in which case
  \begin{align}
    \Succ{2}{{\rm aft}}{x,y} &= \Succ{2}{{\rm bef}}{x,y} \oplus (a,b) \oplus \Succ{2}{{\rm bef}}{a,b},
      \label{eqn:Lemma34e}
  \end{align}
  while all other sets of sets of successors and predecessors of $(s,y)$ and $(a,b)$ remain unchanged.
\end{lemma}

%%%%%%%%%%%%%%%%%%%%%%%%%%%%%%%%%%%%%%%%%%%%%%%
\subsection{Case III: Birth-death Transpositions}
\label{sec:3.3}
%%%%%%%%%%%%%%%%%%%%%%%%%%%%%%%%%%%%%%%%%%%%%%%

Case~III is characterized by the transposition of cells $b$ and $x$, in which $(a,b)$ and $(x,y)$ are disjoint or overlapping birth-death pairs with consecutive dimensions, so $f(a) < f(b) < f(x) < f(y)$ or $f(a) < f(x) < f(b) < f(y)$ and $\dime{a} + 1 = \dime{b} = \dime{x} = \dime{y} -1$.
Similar to Cases~I and II, we prepare the analysis by canceling all pairs whose pivots lie below the rows or to the left of columns $b$ and $x$.
This includes all predecessors of $(a,b)$ and $(x,y)$, so after these cancellations, both are shallow.
Refer to Figure~\ref{fig:CaseBD} for the two sub-cases, with the conditions referring to the complex after the initial cancellations:
\begin{figure}[h!]
  \centering \vspace{0.05in}
  \resizebox{!}{4.6in}{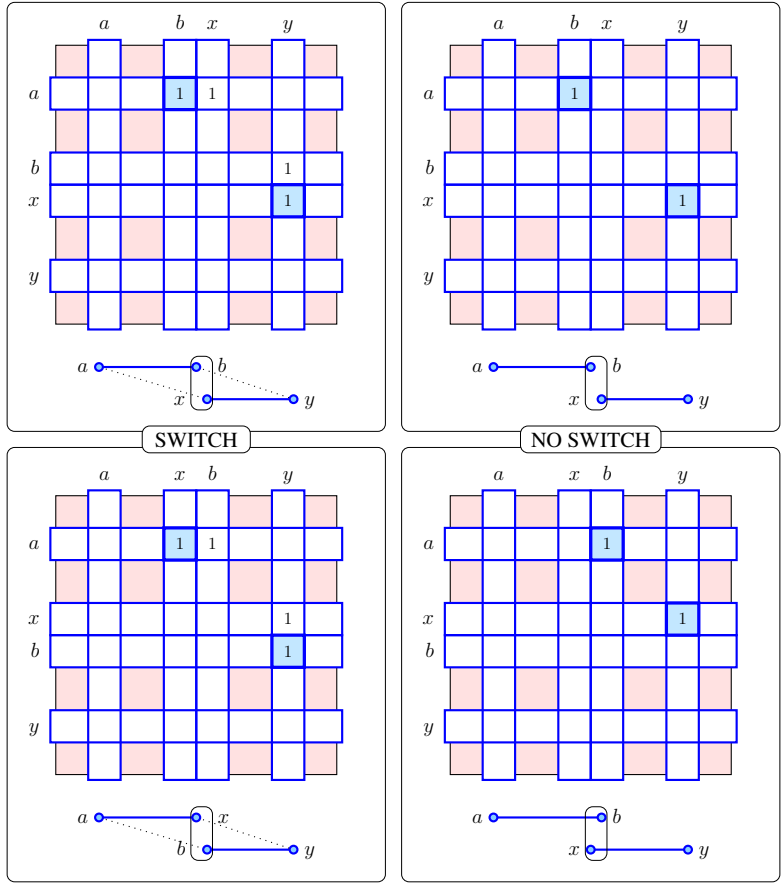}
  \vspace{-0.05in}
  \caption{{\footnotesize The configurations in Cases~III.1 and III.2 on the \emph{left} and \emph{right}, before and after the transposition of a birth-giving cell with a death-giving cell.
  Within each panel, we show the relevant rows and columns of the matrix, and the corresponding birth-death pairs as bars.}}
  \label{fig:CaseBD}
\end{figure}
\begin{description}
  \item[Case III.1:] $a \prec x$ and $b \prec y$ for nested pairs, as illustrated in the upper left panel of Figure~\ref{fig:CaseBD}.
    After transposing $b$ and $x$, $(a,x)$ and $(b,y)$ are both shallow, so the transposition is a switch and the two new pairs are birth-death pairs of the new filter.
    By the first two properties in Lemma~\ref{lem:properties_of_depth_poset}, there is no arc connecting the two pairs in the depth poset.
    The inverse transposition is symmetric and follows the same rules.
  \item[Case III.2:] $a \nprec x$ and $b \nprec y$ for nested pairs, as illustrated in the upper right panel of Figure~\ref{fig:CaseBD}.
    After transposing $b$ and $x$, $(a,b)$ and $(x,y)$ are still shallow, so the transposition is not a switch.
    The inverse transposition is defined by the same conditions but now for two overlapping pairs, as illustrated in the bottom right panel of Figure~\ref{fig:CaseBD}.
\end{description}
There are no additional cases, because the relevant birth-death pairs are shallow both before and after the transposition, which excludes configurations different from the ones illustrated in Figure~\ref{fig:CaseBD}.
Observe that Case~III.1 is quite different from all others, because the transposition changes the types of two cells: $b$ from death- to birth-giving, and $x$ from birth- to death-giving.
Hence, their corresponding rows and columns in the boundary matrix assume different roles, which may have a profound impact on the depth poset.
We prove that this is in fact not the case for the successors, but it can be for the predecessors.

\smallskip
Recall that Algorithm~\ref{alg:column_reduction} records the column reductions for all cells alike, but only the ones for death-giving cells turn into arcs of the death relation.
Symmetrically, Algorithm~\ref{alg:row_reduction} records the row reductions for all cells, but only the ones for birth-giving cells that also belong to birth-death pairs turn into arcs of the birth relation.
We will now make use of the so far neglected information.
Letting $a$ and $x$ be the birth-giving cells of the two pairs in a BD-type switch, we collect the reductions of column $x$ prior to reaching row $a$ as arcs in the death relation, and letting $y$ and $b$ be death-giving cells of these pairs, we collect the reductions of row $b$ prior to reaching column $y$ as arcs in the birth relation.
\begin{lemma}
  \label{lem:changes_after_BD-type_switch}
  Let $f \colon K \to \Rspace$ be a filter on a Lefschetz complex, $(a,b)$ and $(x,y)$ birth-death pairs of consecutive dimensions, and $b,x$ consecutive in the ordering implied by the filter.
  Then $a \prec x$ and $b \prec y$ iff the transposition of $b$ and $x$ is a BD-type switch (see the left panels in Figure~\ref{fig:CaseBD}), and in this case
\small
\begin{align}
\Succ{1}{{\rm aft}}{a,x} &= \Succ{1}{{\rm bef}}{a,b};\qquad\qquad\qquad\qquad\qquad\quad \Succ{1}{{\rm aft}}{b,y} = \Succ{1}{{\rm bef}}{x,y} ;
\label{eqn:Lemma35a} \\
\Succ{2}{{\rm aft}}{a,x} &= \Succ{2}{{\rm bef}}{a,b};\qquad\qquad\qquad\qquad\qquad\quad \Succ{2}{{\rm aft}}{b,y} = \Succ{2}{{\rm bef}}{x,y} ;
\label{eqn:Lemma35b} \\
\Pred{1}{{\rm aft}}{a,x} &= \{(t,x) \mid U_1^{{\rm bef}}[t,x]=1, f(t) > f(a) \};\quad\ \Pred{1}{{\rm aft}}{b,y} = \Pred{1}{{\rm bef}}{x,y};
\label{eqn:Lemma35c} \\
\Pred{2}{{\rm aft}}{a,x} &= \Pred{2}{{\rm bef}}{a,b};\quad\ \Pred{2}{{\rm aft}}{b,y} = \{(b,s) \mid U_2^{{\rm bef}}[b,s]=1, f(s) < f(y) \}.
\label{eqn:Lemma35d}
\end{align}
\normalsize
  Otherwise, the depth relation remains as is (see the right panels in Figure~\ref{fig:CaseBD}).
\end{lemma}
\begin{proof}
  To prove \eqref{eqn:Lemma35a} and \eqref{eqn:Lemma35b}, we first consider rows $b$ and $x$ in the boundary matrix (after canceling all birth-death pairs whose pivots are below row $x$ or to the left of column $b$).
  Before the transposition, $b$ gives death, so there is no pivot in its row.
  We claim that this implies that the two rows are equal.
  Indeed, if there is a column in which the entries in rows $b$ and $x$ are different, then the leftmost such entry would be a pivot, which is a contradiction to the assumptions.
  Since the two rows are the same, transposing them makes no difference, other than exchanging their names.
  This implies $\Succ{1}{{\rm aft}}{a,x} = \Succ{1}{{\rm bef}}{a,b}$ as well as $\Succ{1}{{\rm aft}}{b,y} = \Succ{1}{{\rm bef}}{x,y}$.
  The symmetric argument implies the corresponding relations for the successors in the birth relation.

  \smallskip
  To prove \eqref{eqn:Lemma35c} and \eqref{eqn:Lemma35d}, we recall that rows $b$ and $x$ are equal, so in the death relation the predecessors of $(x,y)$ before the transposition are the predecessors of $(b,y)$ after the transposition and, symmetrically, in the birth relation the predecessors of $(a,b)$ before the transposition are the predecessors of $(a,x)$ after the transposition.
  The remaining predecessors of $(a,x)$ and $(b,y)$ are determined while canceling the birth-death pairs with pivots below and to the left of the rows and columns of these pairs.
  We can get this information from column $x$ in $U_1$ and row $b$ in $U_2$ as computed before the transposition.
\end{proof}
The symmetric cases---in which we read the panels in Figure~\ref{fig:CaseBD} from bottom to top---are similar.
This completes the case analysis of changes to the depth poset caused by the transposition of two cells that are consecutive in the filter.

%% \newpage
%%%%%%%%%%%%%%%%%%%%%%%%%%%%%%%%%%%%%%%%%%%%%%%
%%%%%%%%%%%%%%%%%%%%%%%%%%%%%%%%%%%%%%%%%%%%%%%
\section{Computational Experiments}
\label{sec:4}
%%%%%%%%%%%%%%%%%%%%%%%%%%%%%%%%%%%%%%%%%%%%%%%
%%%%%%%%%%%%%%%%%%%%%%%%%%%%%%%%%%%%%%%%%%%%%%%

In this section, we present measurements of the depth poset collected while running the algorithms on $1$-, $2$-, and $3$-dimensional random functions and straight-line homotopies between them.
For each set of parameters, the results are averaged over ten repeats.

%%%%%%%%%%%%%%%%%%%%%%%%%%%%%%%%%%%%%%%%%%%%%%%
\subsection{Random Piecewise Linear Functions}
\label{sec:4.1}
%%%%%%%%%%%%%%%%%%%%%%%%%%%%%%%%%%%%%%%%%%%%%%%

To minimize complications and avoid boundary effects as well as any essential homology, we construct the random functions on a regular cubical subdivision of the $d$-dimensional torus.
Choosing a positive integer, $n$, we use $(\Rspace / n \Zspace)^d$ as a model of the $d$-torus, subdivide it with the unit $d$-cubes of $(\Zspace / n \Zspace)^d + [0,1]^d$, and kill its essential homology classes by adding $2^d$ \emph{extra cells}, namely the empty set as the sole $(-1)$-cell and $\binom{d}{p}$ $(p+1)$-cells, each  with boundary equal to a subdivided $p$-torus, for $1 \leq p \leq d$.
Write $K = K(n,d)$ for the resulting complex, which consists of $\binom{d}{p} n^d$ $p$-cubes, for $0 \leq p \leq d$, plus the $2^d$ extra cells.

\begin{figure}[hbt]
  \centering
  \vspace{-0.05in}
  \hspace{-0.1in} \includegraphics[width=0.35\textwidth]{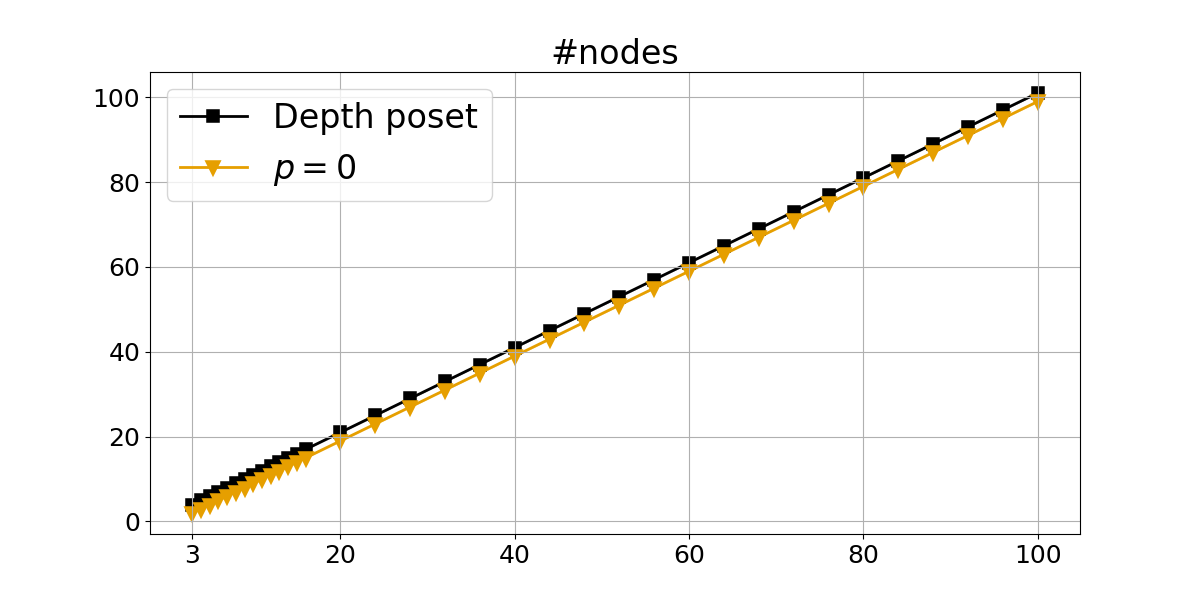}
  \hspace{-0.2in} \includegraphics[width=0.35\textwidth]{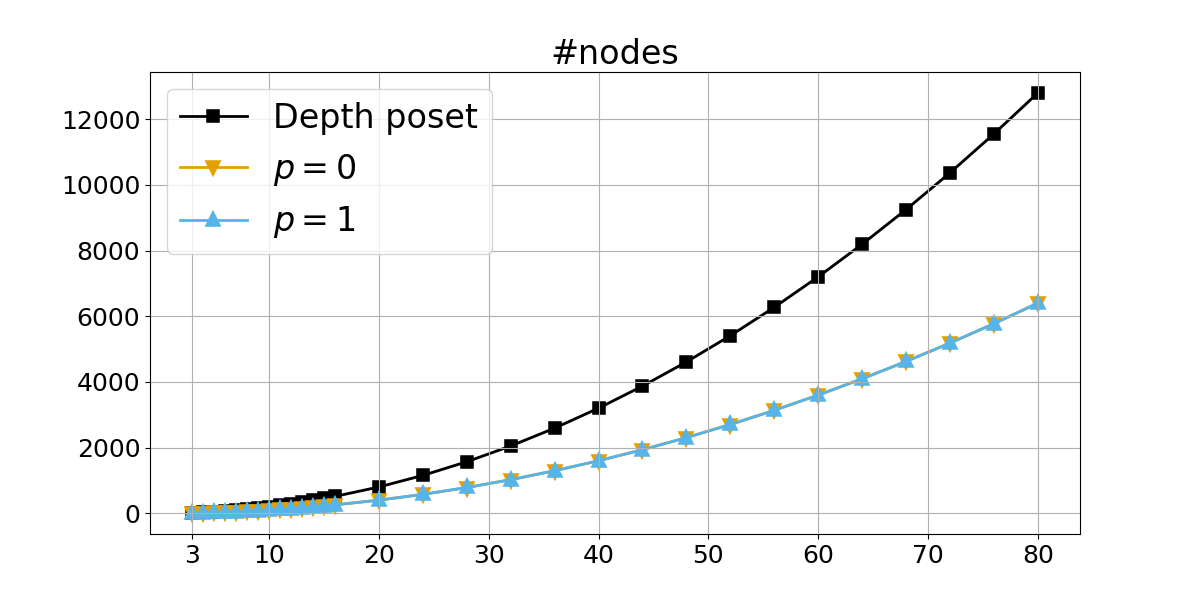}
  \hspace{-0.2in} \includegraphics[width=0.35\textwidth]{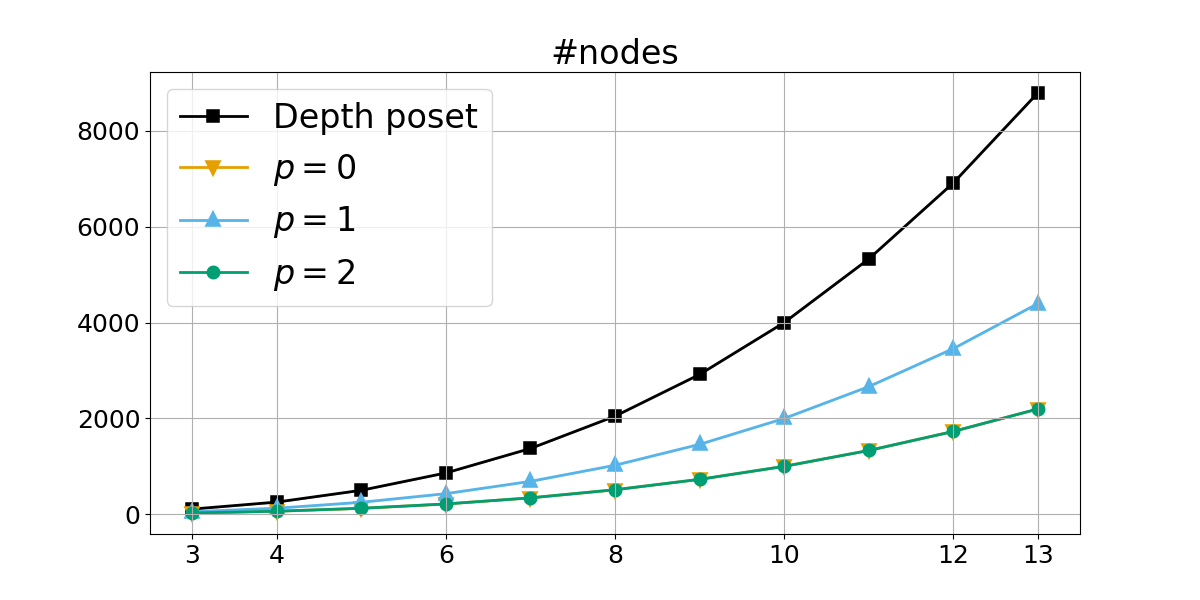} \\
  \vspace{-0.00in}
  \hspace{-0.1in} \includegraphics[width=0.35\textwidth]{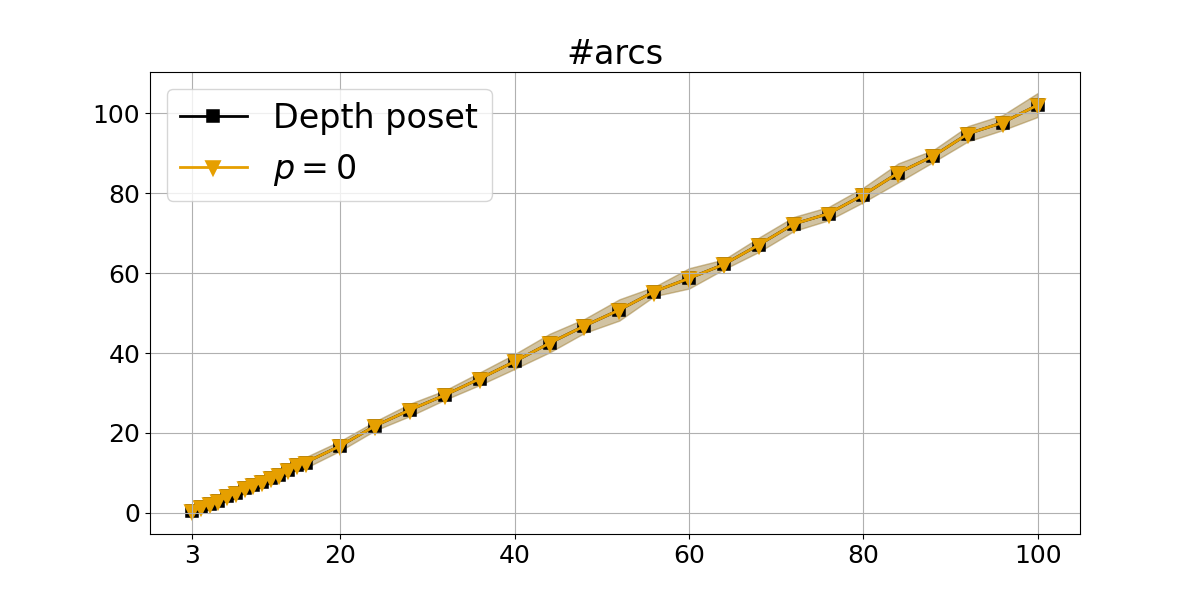}
  \hspace{-0.2in} \includegraphics[width=0.35\textwidth]{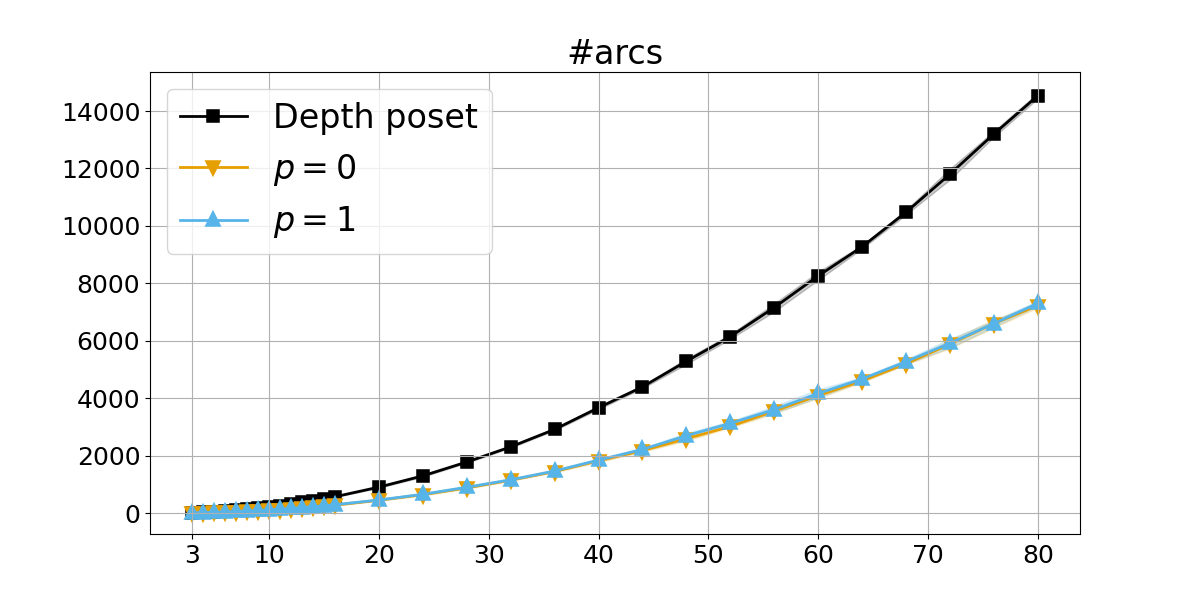}
  \hspace{-0.2in} \includegraphics[width=0.35\textwidth]{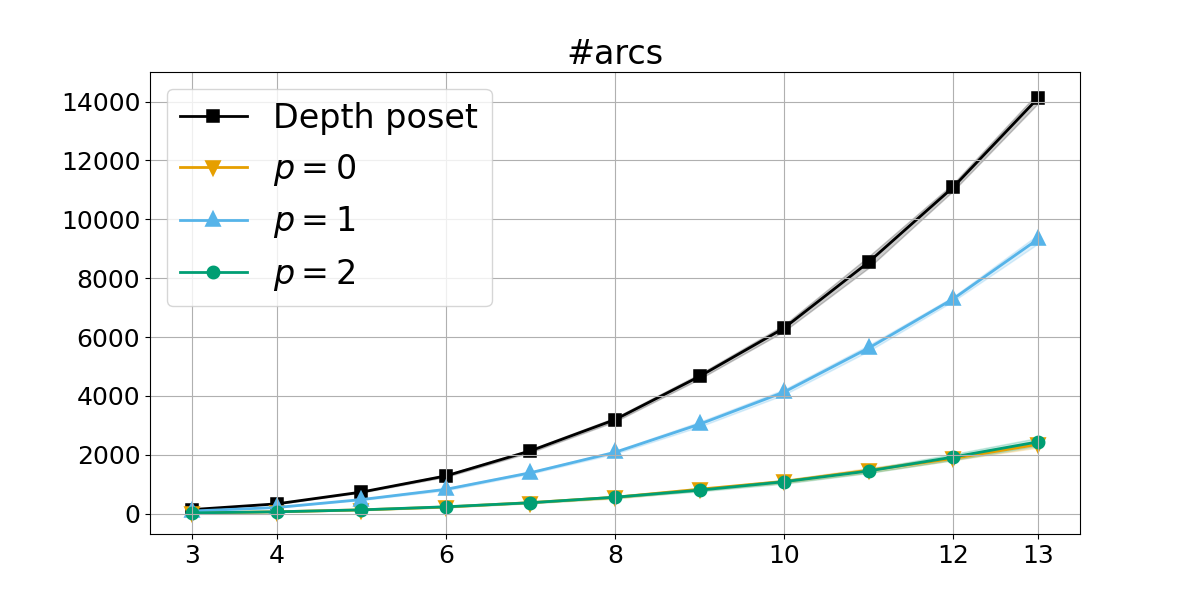} \\
  \vspace{-0.05in}
  \caption{\footnotesize \emph{From left to right:} the average number of nodes (\emph{upper row}) and arcs (\emph{lower row}) in the transitive reduction of the depth poset of a random function on the $d$-torus, for $d = 1, 2, 3$, respectively.
  Because of the symmetries in the construction, the curves for $p=0$ and $p=d-1$ are nearly identical.}
  \label{fig:nodes_arcs}
\end{figure}
To turn this complex into a random function, $f \colon K \to \Rspace$, we choose $f(\gamma)$ in $[p,p+1]$, whenever $p = \dime{\gamma}$.
In particular, we choose $f(\gamma)$ uniformly at random in $[p,p+1]$ whenever $\gamma$ is a $p$-cube, and we set $f(\gamma) = p+1$ whenever it is an extra $p$-cell.
In spite of the extra cells that kill the essential homology of the $d$-torus, we will refer to $f$ as a \emph{random function on the $d$-torus}.
\begin{figure}[hbt]
  \centering
  \vspace{-0.05in}
  \hspace{-0.1in} \includegraphics[width=0.35\textwidth]{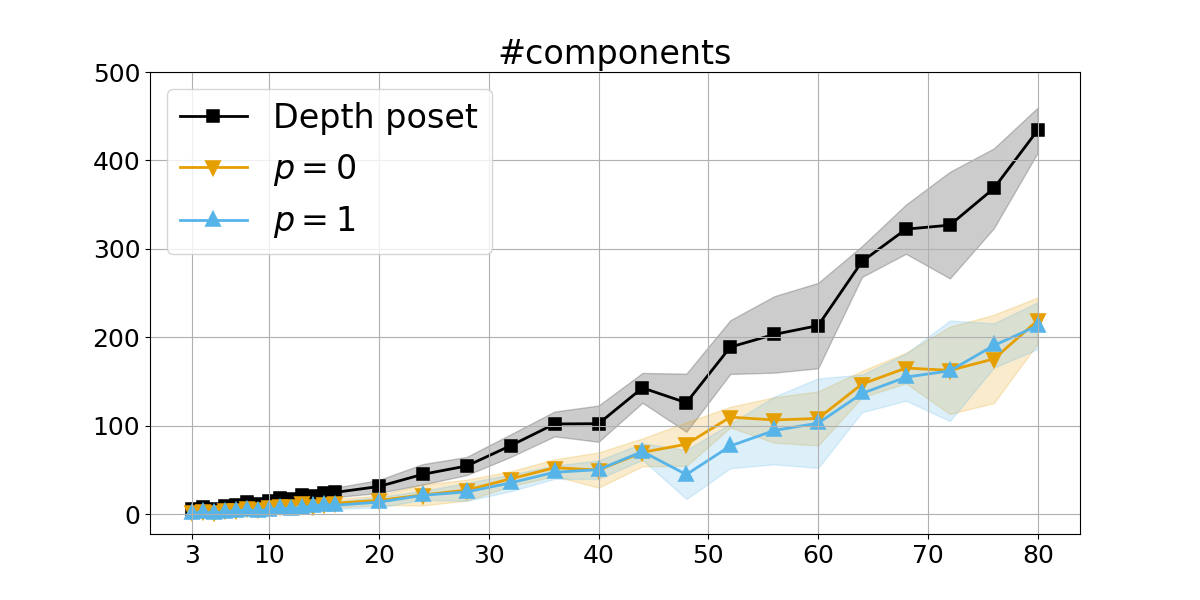}
  \hspace{-0.2in} \includegraphics[width=0.35\textwidth]{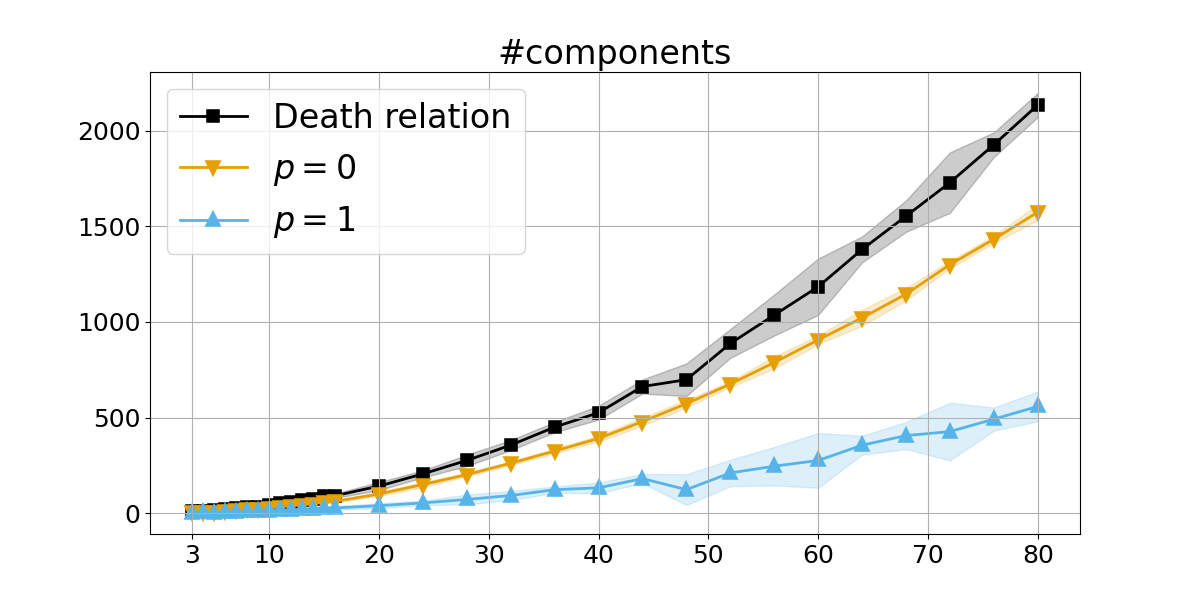}
  \hspace{-0.2in} \includegraphics[width=0.35\textwidth]{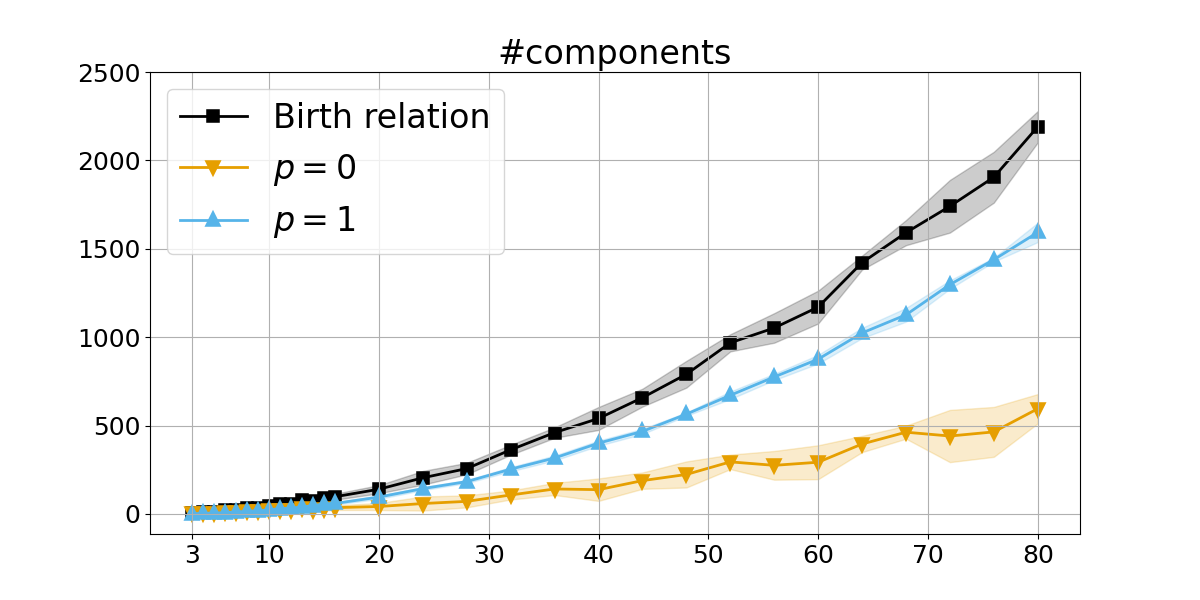} \\
  \vspace{-0.00in}
  \hspace{-0.1in} \includegraphics[width=0.35\textwidth]{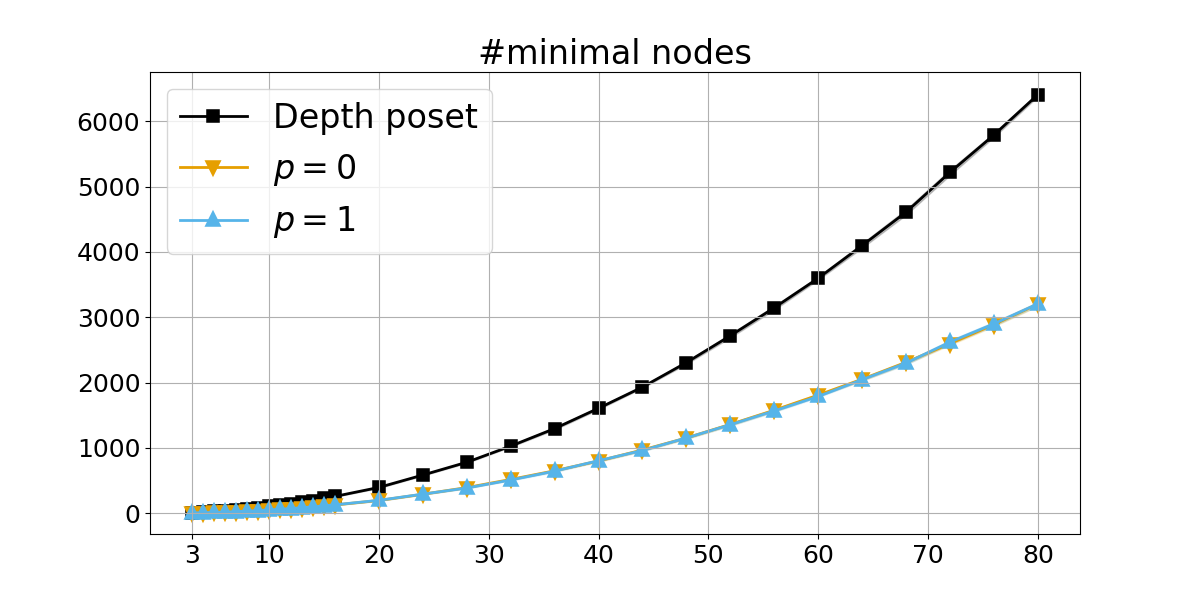}
  \hspace{-0.2in} \includegraphics[width=0.35\textwidth]{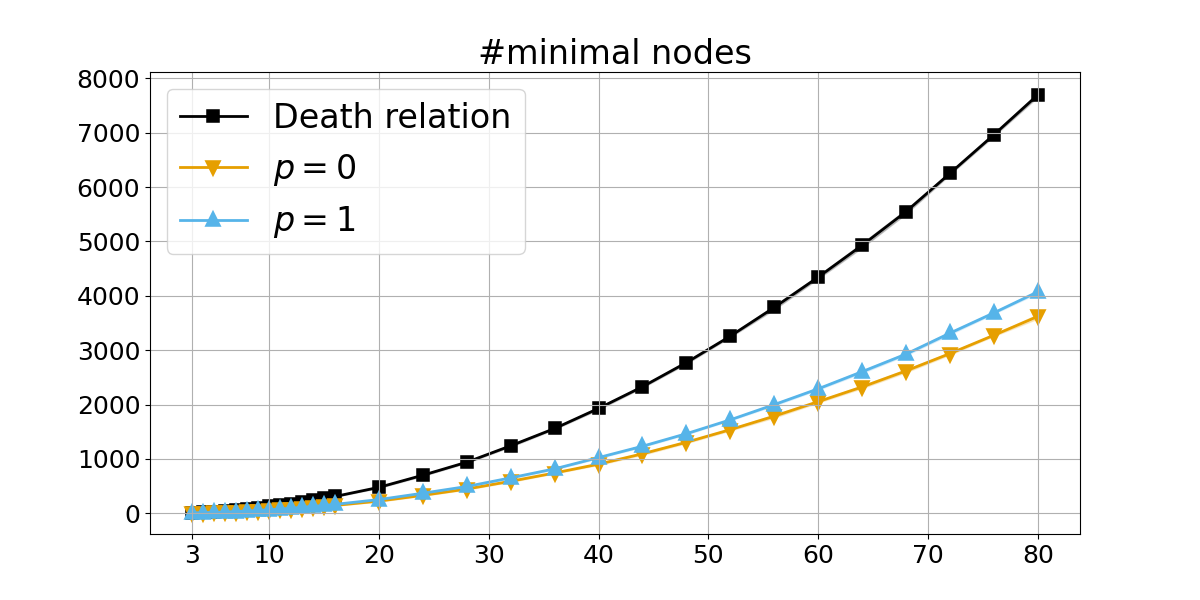}
  \hspace{-0.2in} \includegraphics[width=0.35\textwidth]{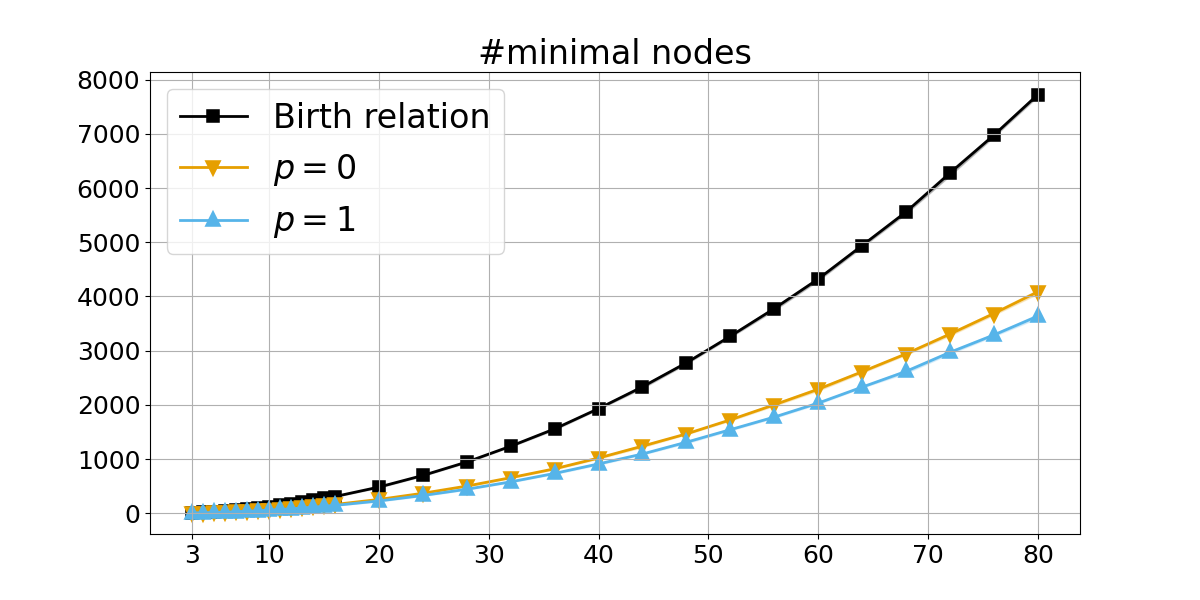} \\
  \vspace{-0.00in}
  \hspace{-0.1in} \includegraphics[width=0.35\textwidth]{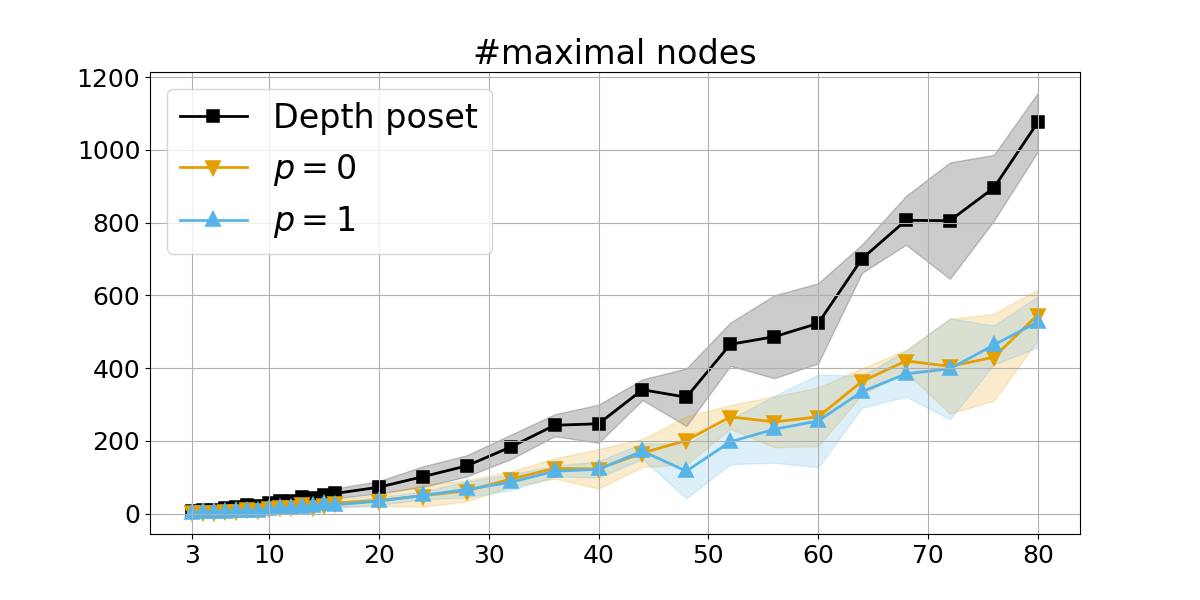}
  \hspace{-0.2in} \includegraphics[width=0.35\textwidth]{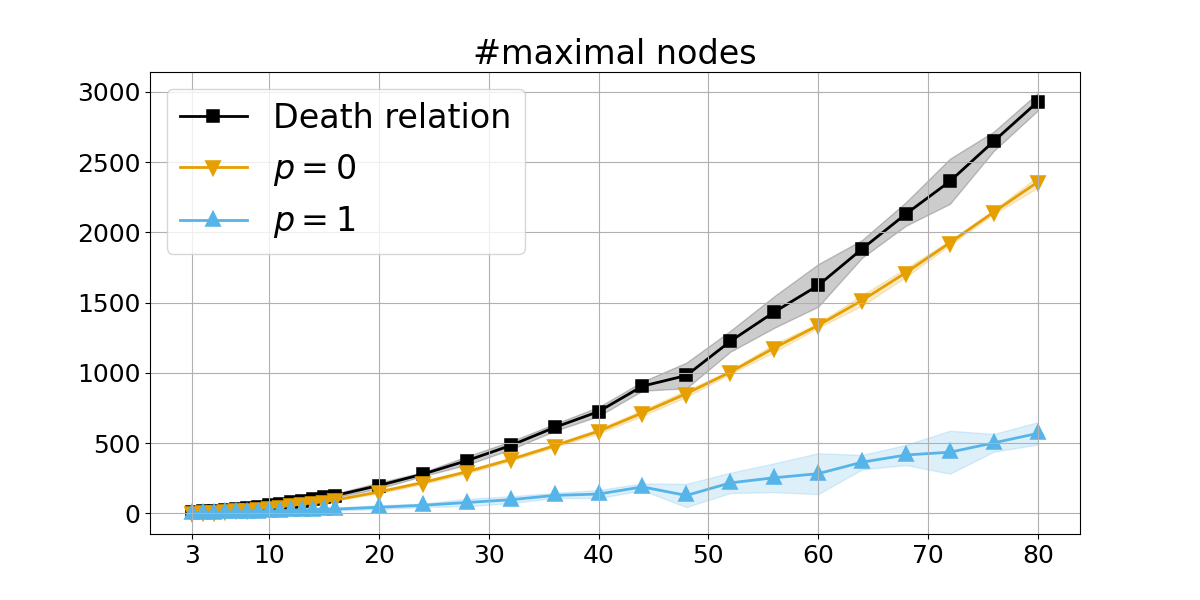}
  \hspace{-0.2in} \includegraphics[width=0.35\textwidth]{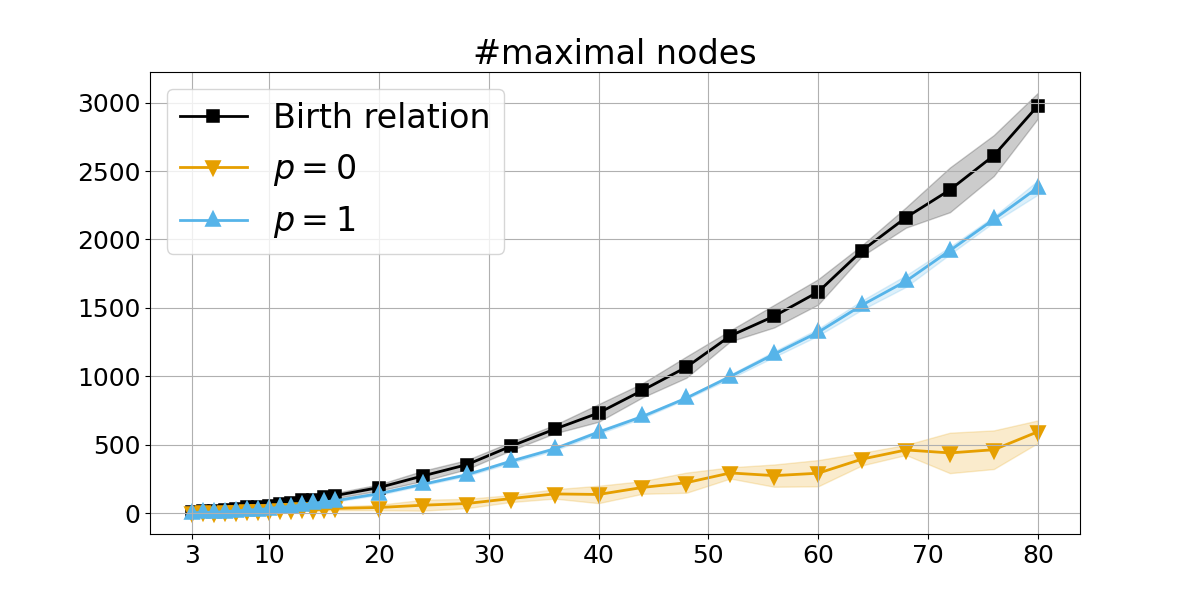} \\
  \vspace{-0.05in}
  \caption{\footnotesize \emph{From left to right:} the average numbers of components, minimal nodes, and maximal nodes of the depth poset, the death relation, and the birth relation of a random function on the $2$-torus, respectively.
  Because of symmetry, the curves for the death and birth relations are nearly identical.}
  \label{fig:components_min_max}
\end{figure}
By construction, the filter that orders the cells by value is also ordered by dimension.
We may think of $f$ as a piecewise constant function, but there is an alternative interpretation as a piecewise linear function.
Letting $\hat{K} = \hat{K} (n,d)$ be the barycentric subdivision of $K$, each vertex $\hat{\gamma} \in \hat{K}$ is the center of a cell $\gamma \in K$, and we map $\hat{\gamma}$ to $\hat{f} (\hat{\gamma}) = f(\gamma)$.
We then get $\hat{f}$ on the underlying space of $K$ by piecewise linear interpolation of its values at the vertices.
By construction, each center of a $p$-cell in $K$ is a critical point of index $p$ of $\hat{f}$.
All critical points are paired up as nodes of the depth poset, which implies that $n$ and $d$ determine the number of nodes for each $-1 \leq p \leq d$; see the upper row in Figure~\ref{fig:nodes_arcs}.
In contrast, the number of arcs in the depth poset has a small but non-zero variance; see the lower row of Figure~\ref{fig:nodes_arcs}.

\smallskip
As proved in \cite{ELMS25}, the depth poset splits by the dimensions of the nodes, but depending on the data it may split further.
The upper row in Figure~\ref{fig:components_min_max} shows that it indeed does so for random functions on the $2$-torus.
Each component has at least one minimal and one maximal node, but there may be more, and the middle and lower rows of the same figure support this possibility.
Recall that the minimal nodes of the depth poset are exactly the shallow pairs, which explains why their number has barely any variance.
In $d=2$ dimensions, there are $4 n^2$ incident vertex-edge pairs, and such a pair is shallow iff the edge precedes the other three incident edges of the vertex, and the vertex succeeds the other incident vertex of the edge.
Ignoring the extra $(-1)$-cell, the chance of such a pair to be shallow is therefore~ $\frac{1}{8}$.
Symmetrically, there are $4n^2$ incident edge-square pairs, and ignoring the extra $3$-cell, the chance of an incident edge-square pair to be shallow is again~$\frac{1}{8}$.
This implies that the expected number of shallow pairs is $n^2 + O(1)$; see the first panel in the middle row of Figure~\ref{fig:components_min_max}.
Each shallow pair is also a minimal node in the death and birth relations, but both have additional minimal nodes; see the symmetry of the curves in the second and third panels in the middle row of Figure~\ref{fig:components_min_max}, which is a consequence of the symmetry of cubically subdividing the torus.
We observe that there are many more minimal nodes than there are maximal nodes.
Curiously, the depth poset has many fewer maximal nodes than the death and birth relations, while the difference seems less dramatic for the minimal nodes.

By definition, a poset has no directed cycle, but it may have undirected cycles (pairs of paths that connect the same two nodes).
For a $1$-dimensional function, the death and birth relations have no undirected cycles either but, perhaps surprisingly, the union of these two relations can have such cycles.
Figure~\ref{fig:height_cycles} shows the average number of (undirected) cycles (which we compute as $\#{\rm components} + \#{\rm arcs} - \#{\rm nodes}$), together with the average height (defined as the number of arcs in the longest directed path in the relation). 
\begin{figure}[hbt]
  \centering
  \vspace{-0.05in}
  \hspace{-0.1in} \includegraphics[width=0.35\textwidth]{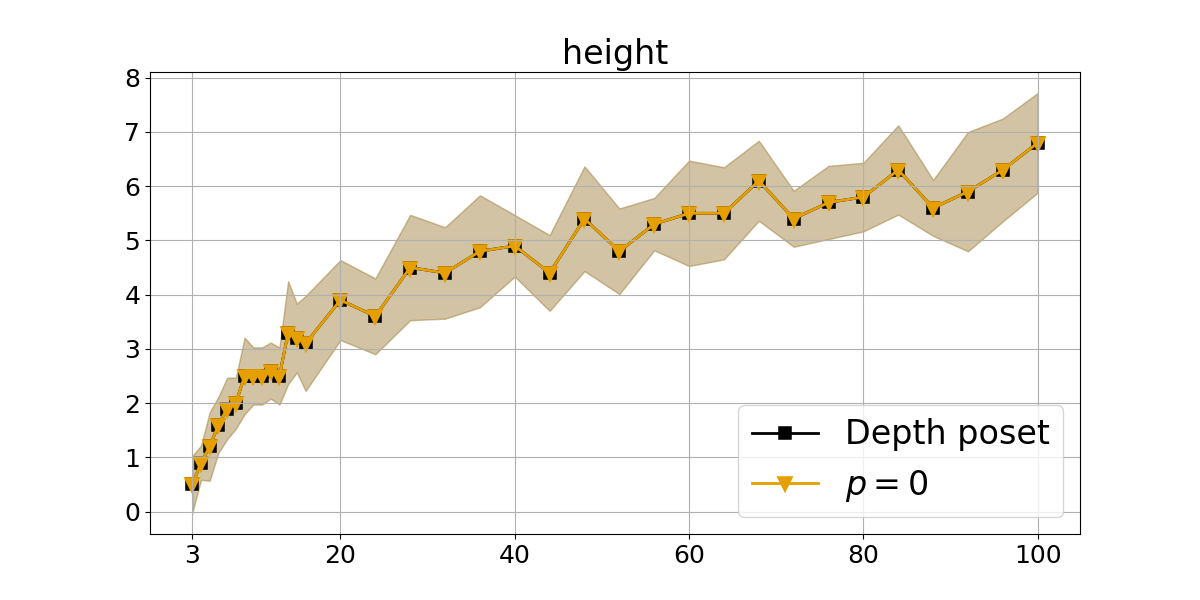}
  \hspace{-0.2in} \includegraphics[width=0.35\textwidth]{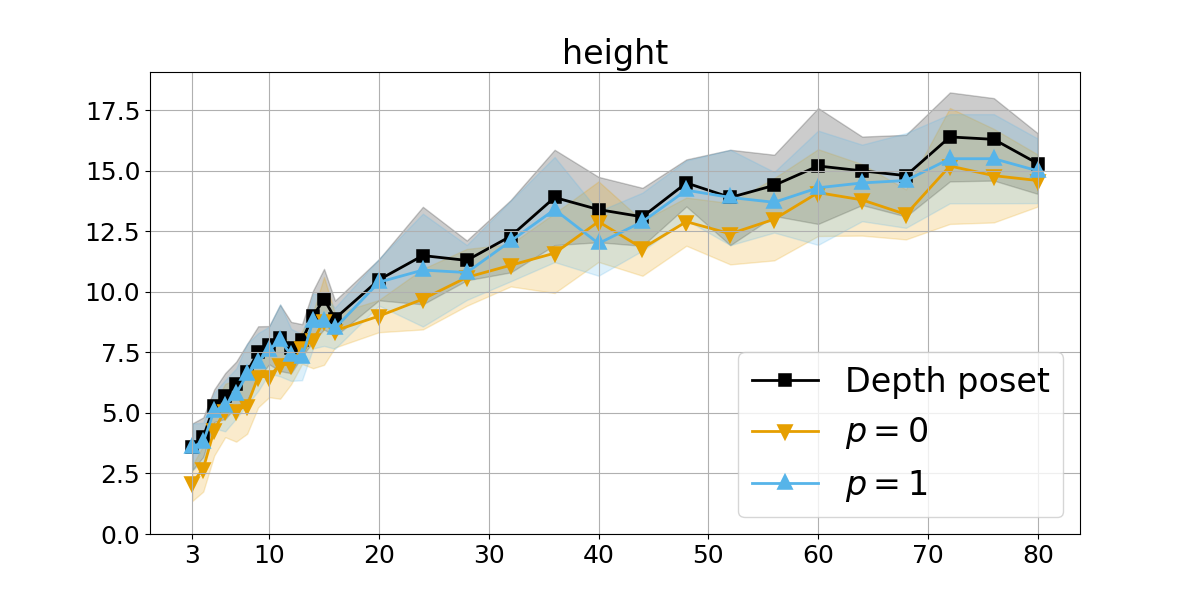}
  \hspace{-0.2in} \includegraphics[width=0.35\textwidth]{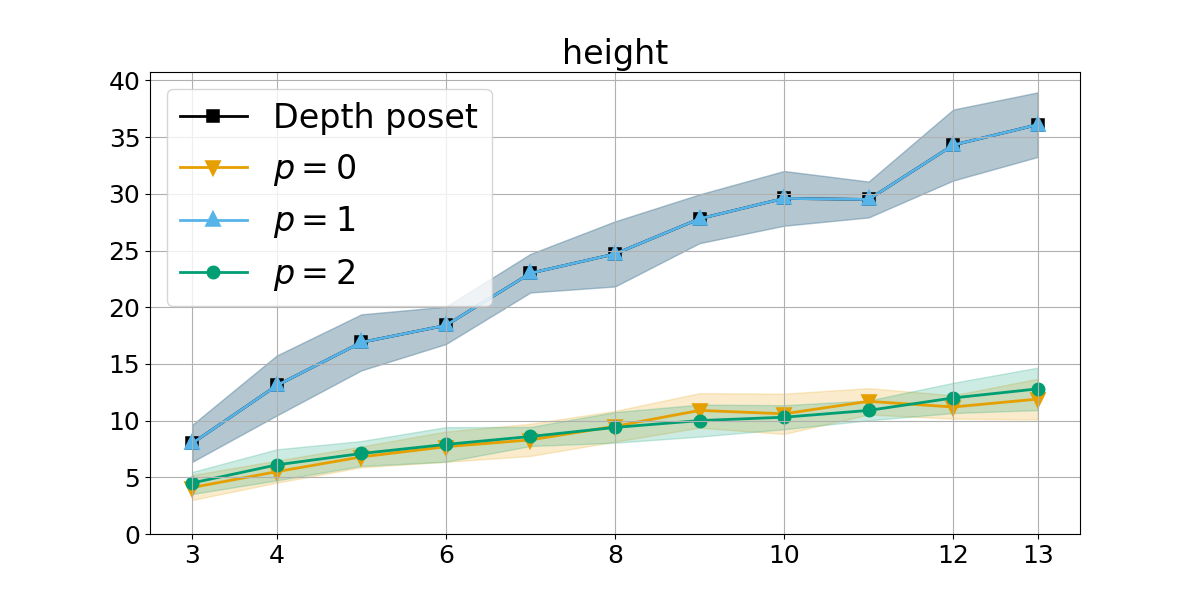} \\
  \vspace{-0.00in}
  \hspace{-0.1in} \includegraphics[width=0.35\textwidth]{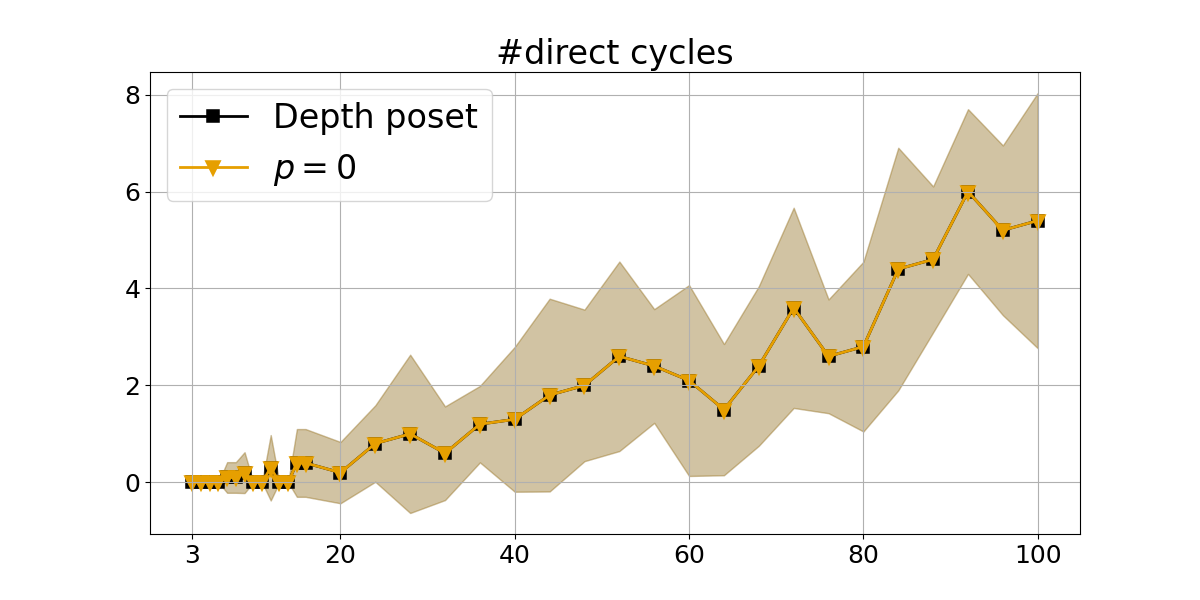}
  \hspace{-0.2in} \includegraphics[width=0.35\textwidth]{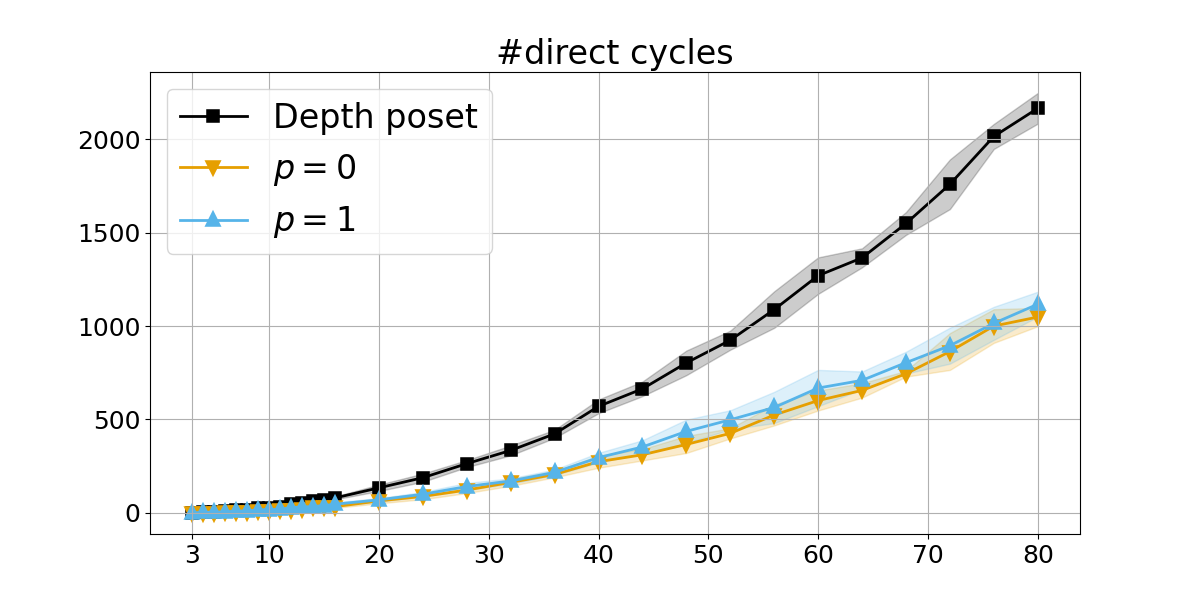}
  \hspace{-0.2in} \includegraphics[width=0.35\textwidth]{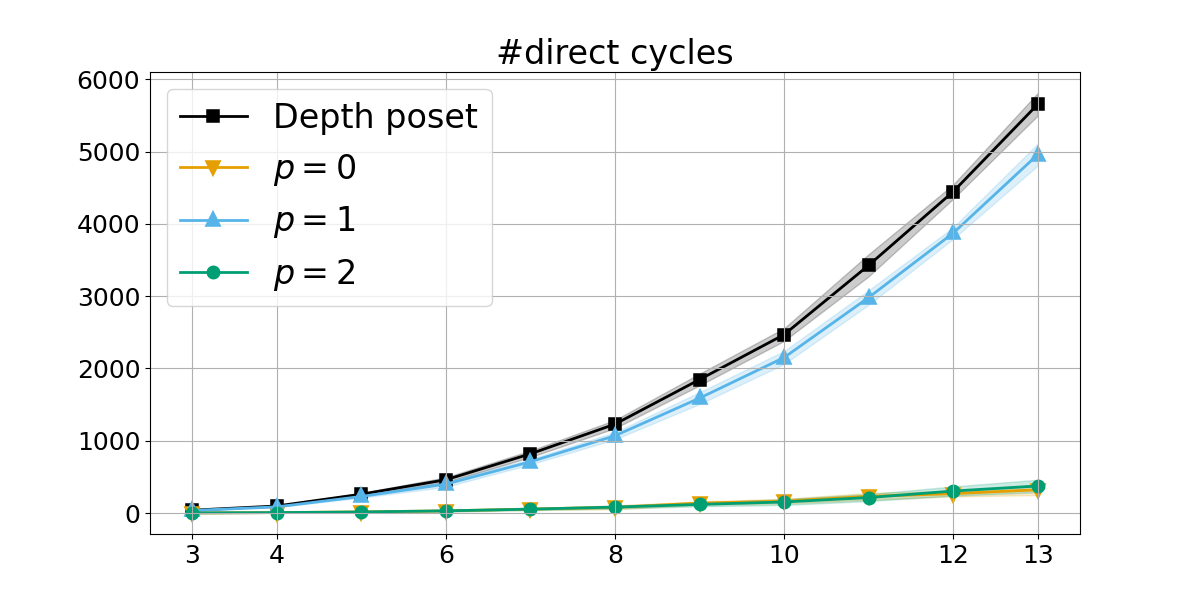} \\
  \vspace{-0.05in}
  \caption{\footnotesize \emph{From left to right:} the average height (\emph{upper row}) and the average number of undirected cycles (\emph{lower row}) in the depth poset of a random function on the $d$-torus, for $d = 1, 2, 3$, respectively.}
  \label{fig:height_cycles}
\end{figure}

%%%%%%%%%%%%%%%%%%%%%%%%%%%%%%%%%%%%%%%%%%%%%%%
\subsection{Straight-line Homotopies}
\label{sec:4.2}
%%%%%%%%%%%%%%%%%%%%%%%%%%%%%%%%%%%%%%%%%%%%%%%

Letting $f_0, f_1 \colon K \to \Rspace$ be two random functions on the $d$-torus, the \emph{straight-line homotopy} between them is the $1$-parameter family of functions $f_\lambda (\gamma) = (1-\lambda) f_0(\gamma) + \lambda f_1(\gamma)$, for all $\gamma \in K$ and $0 \leq \lambda \leq 1$.
To collect the statistics, we perform the transpositions of cells in sequence and one at a time.
Given two cells of not necessarily the same dimension, $\gamma, \eta \in K$, the parameter at which their values match is
\begin{align}
  \lambda (\gamma, \eta) &= \frac{f_0(\eta)-f_0(\gamma)}
                        {[f_0(\eta)-f_0(\gamma)]-[f_1(\eta)-f_1(\gamma)]}.
\end{align}
The relevant such parameter values are the ones in $[0,1]$.
Instead of sorting them, it is more efficient to maintain the filter, and at each step transpose the two cells whose values match next.
If this pair is not unique, we break ties and transpose the cells in sequence.
\begin{figure}[hbt]
  \centering
  \vspace{-0.00in}
  \hspace{-0.1in} \includegraphics[width=0.29\textwidth]{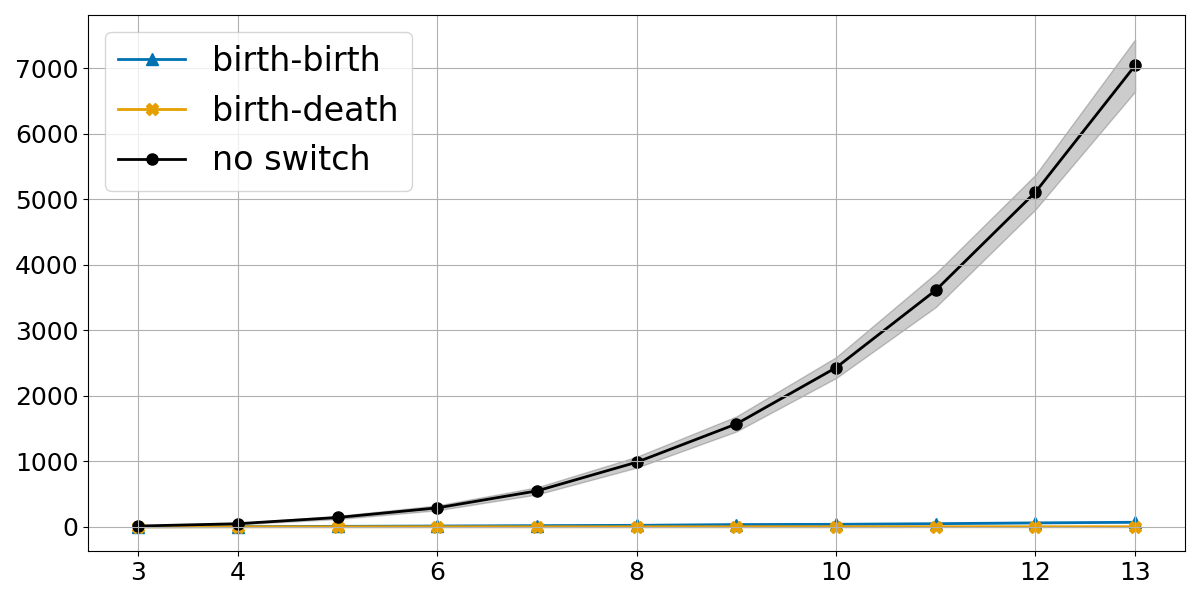}
  \hspace{0.04in} \includegraphics[width=0.29\textwidth]{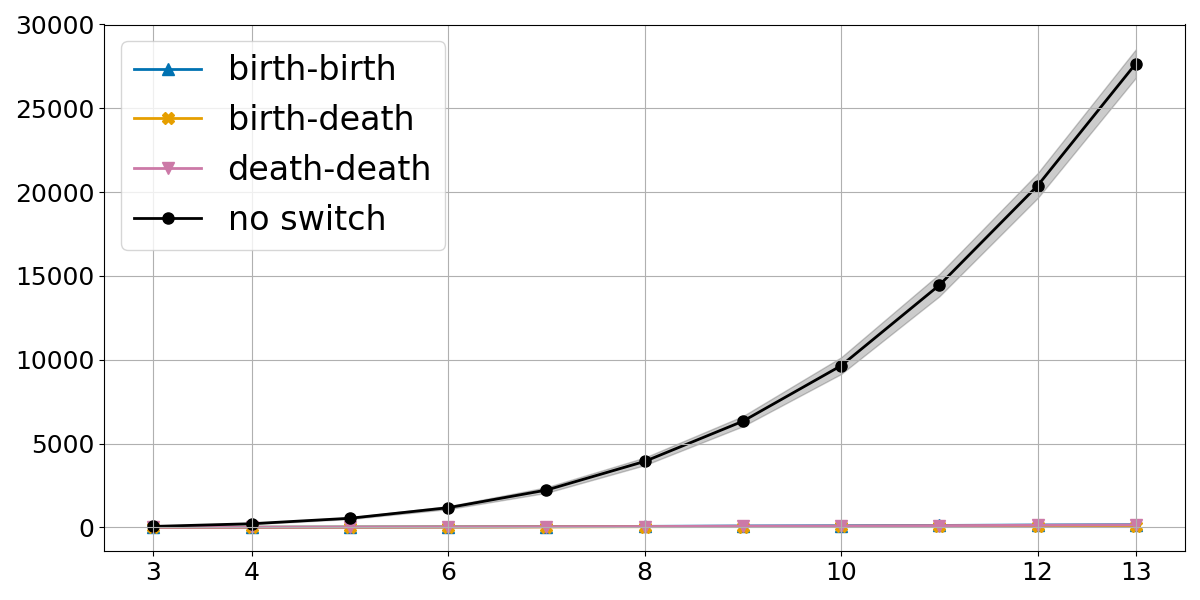}
  \hspace{0.08in} \includegraphics[width=0.29\textwidth]{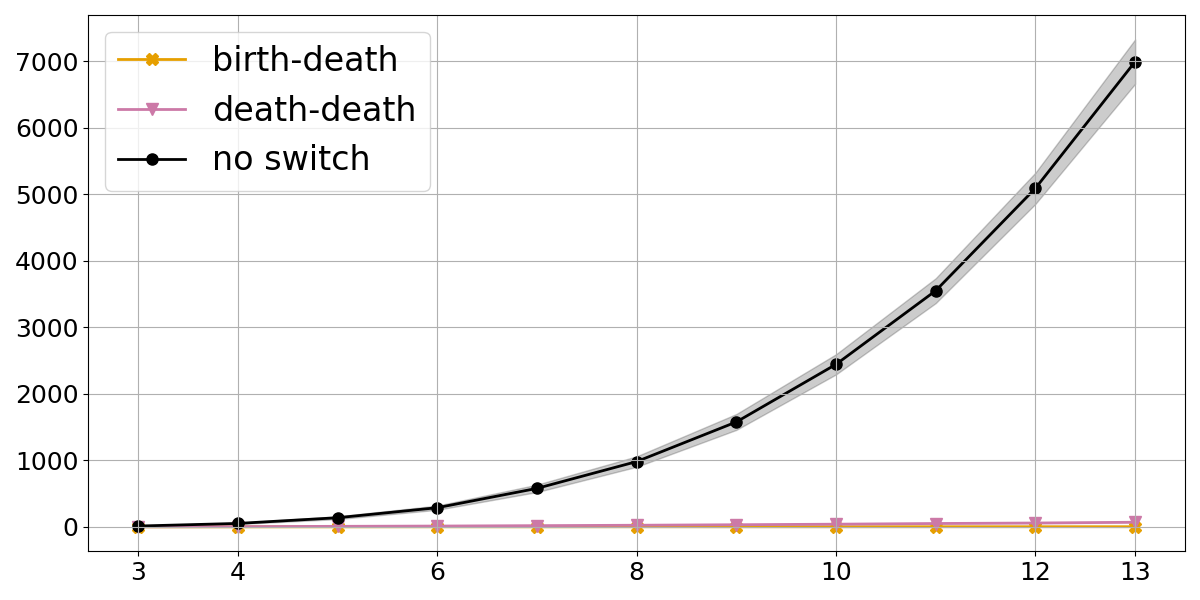} \\
  \vspace{-0.00in}
  \hspace{-0.07in} \includegraphics[width=0.28\textwidth]{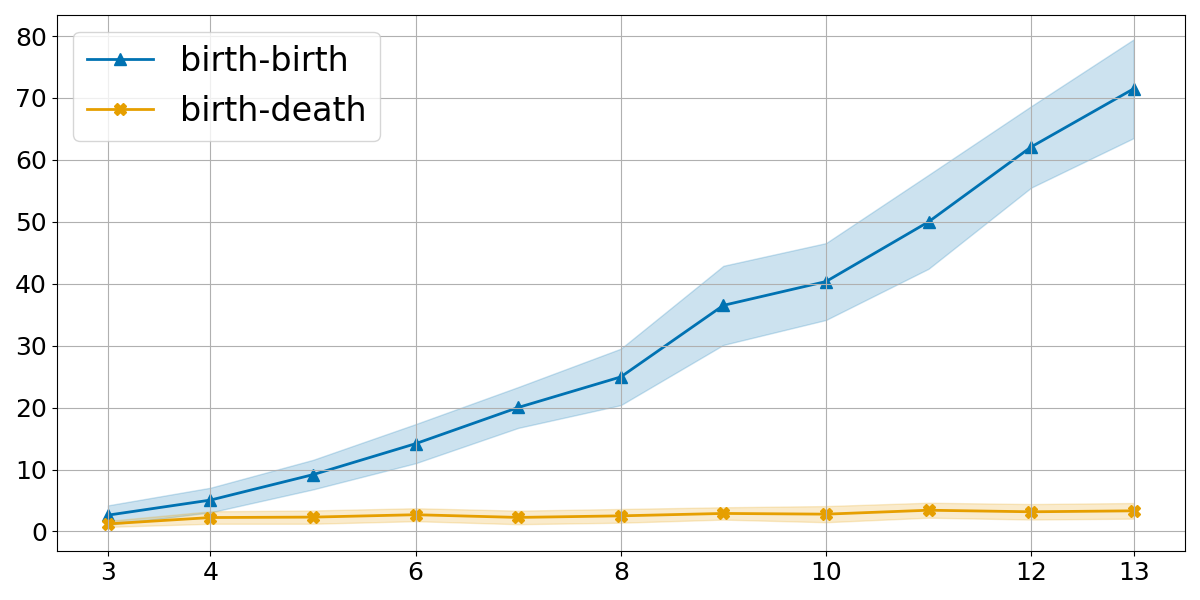}
  \hspace{0.11in} \includegraphics[width=0.28\textwidth]{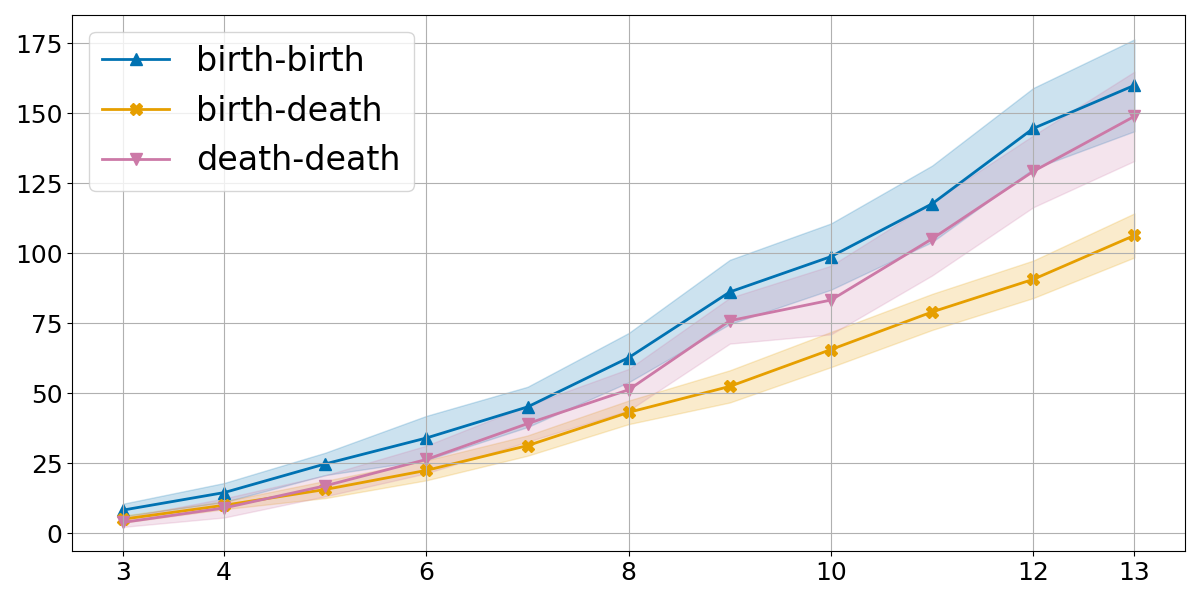}
  \hspace{0.13in} \includegraphics[width=0.28\textwidth]{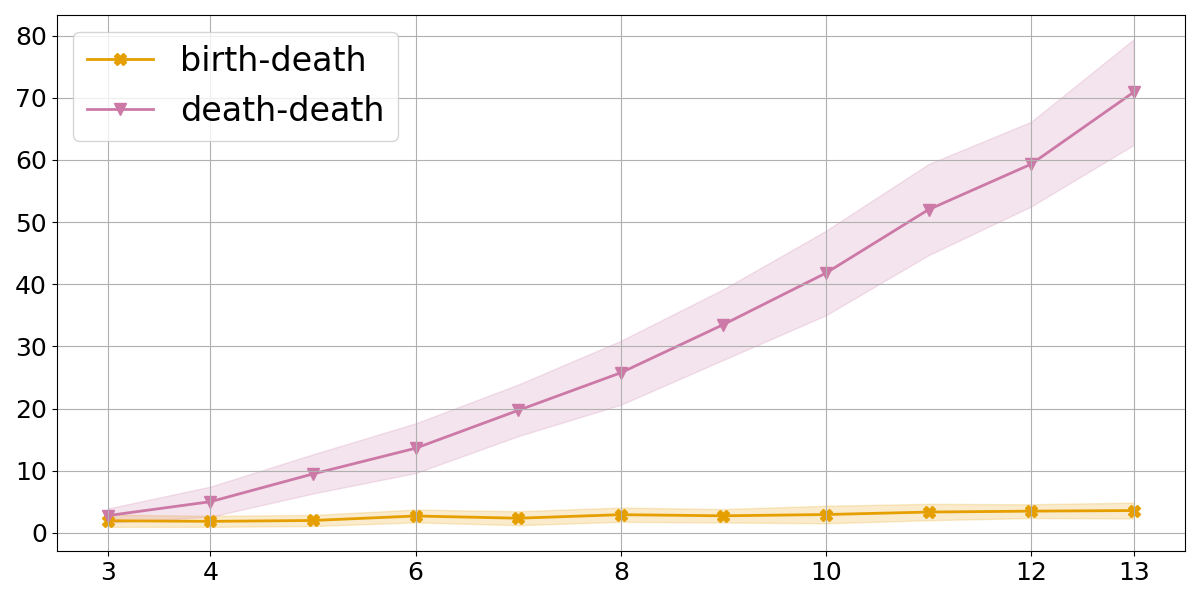} \\
  \vspace{-0.05in}
  \caption{\footnotesize \emph{From left to right:} the average number of transpositions (\emph{upper row}) and switches (\emph{lower row}) between pairs of vertices, edges, and $2$-cells for the straight-line homotopy between two random functions on the $2$-torus, respectively.}
  \label{fig:transposition}
\end{figure}
\begin{figure}[ht]
  \centering
  \vspace{-0.0in}
  \hspace{-0.1in} \includegraphics[width=0.28\textwidth]{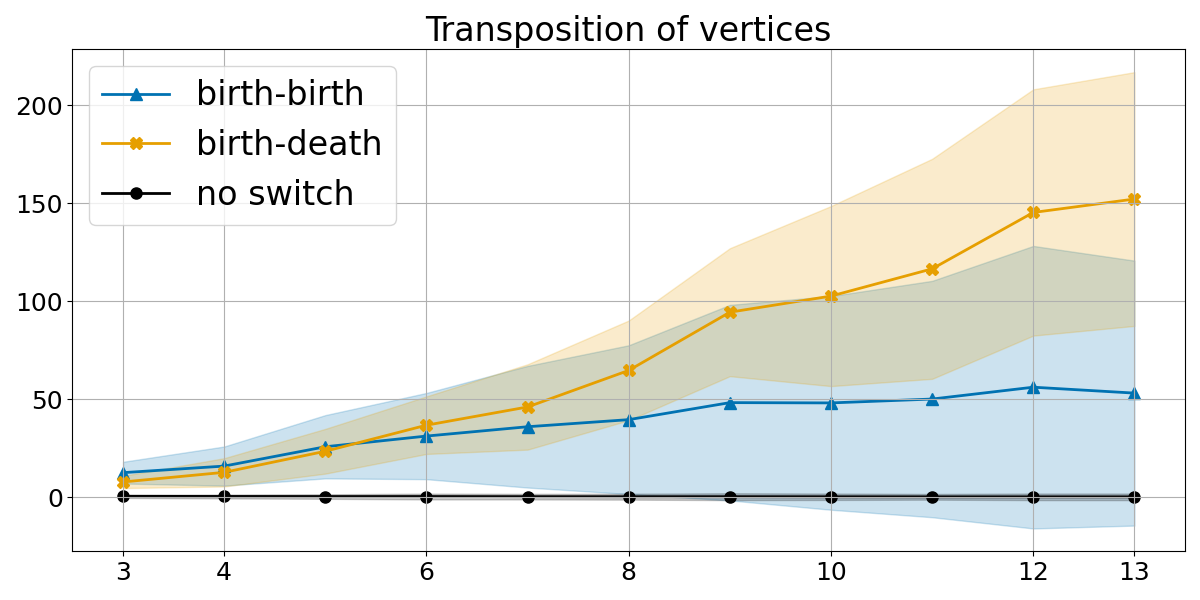}
  \hspace{0.1in} \includegraphics[width=0.28\textwidth]{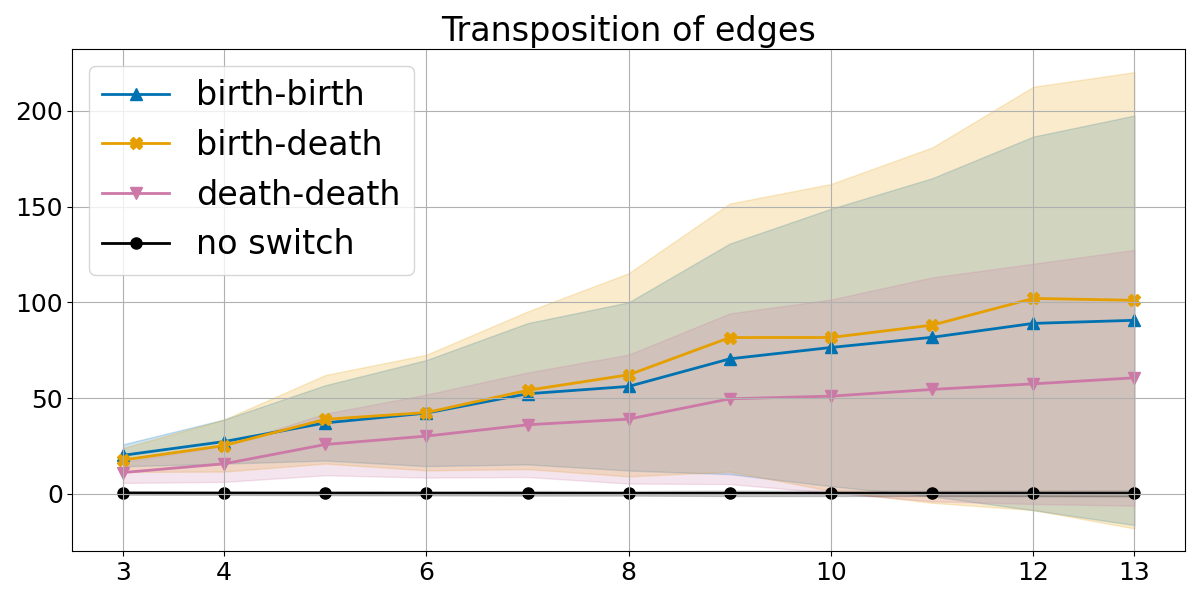}
  \hspace{0.1in} \includegraphics[width=0.28\textwidth]{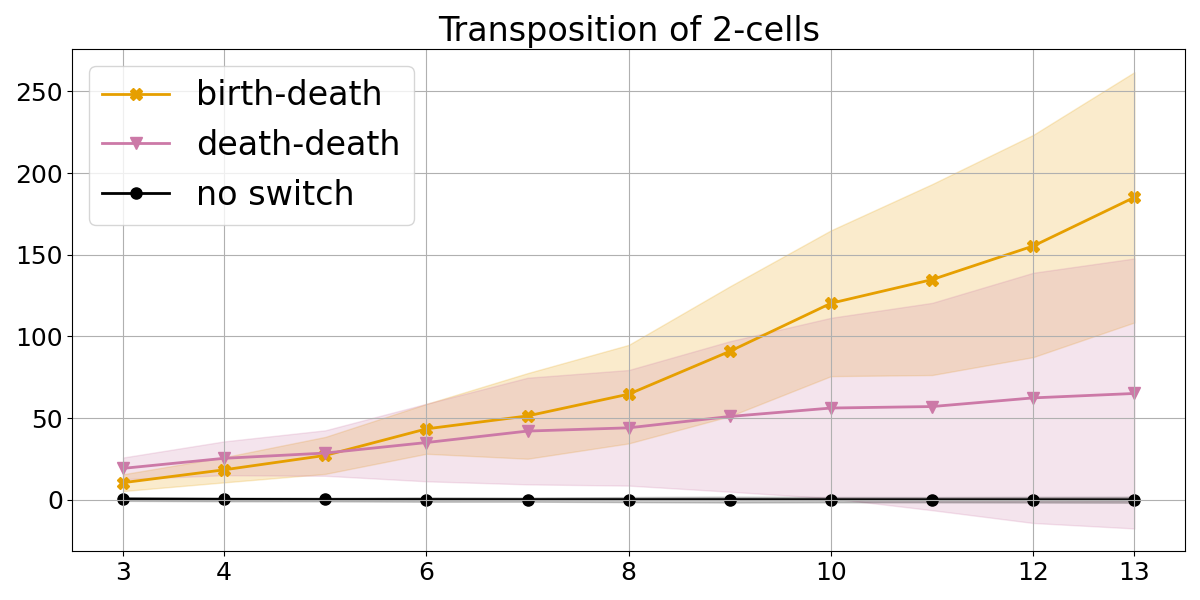} \\
  \vspace{-0.00in}
  \hspace{-0.1in} \includegraphics[width=0.28\textwidth]{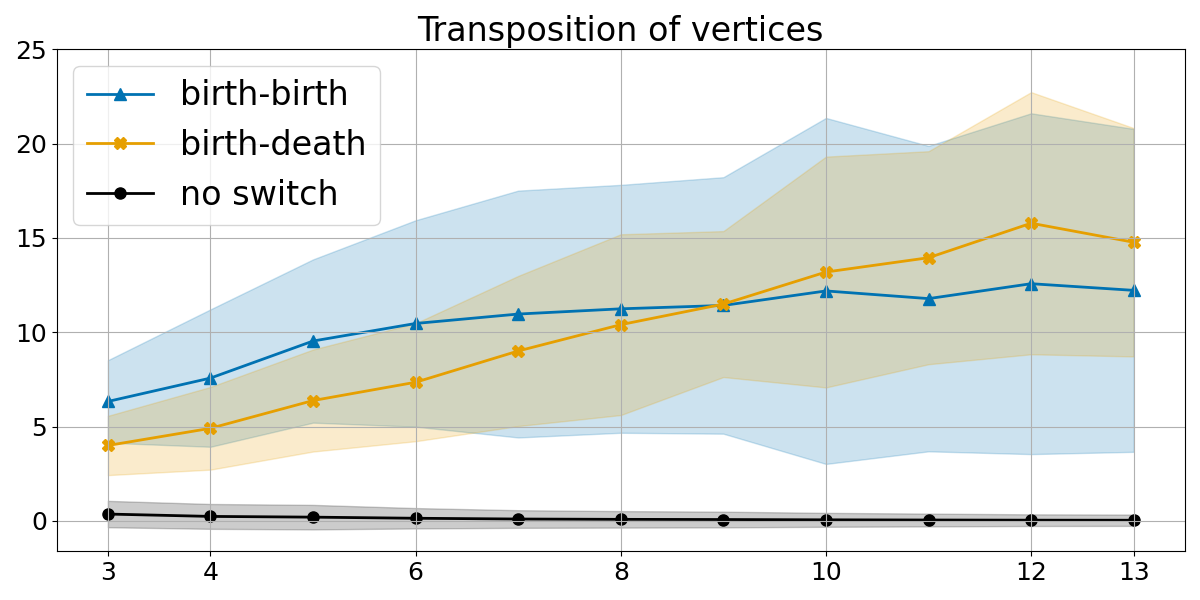}
  \hspace{0.1in} \includegraphics[width=0.28\textwidth]{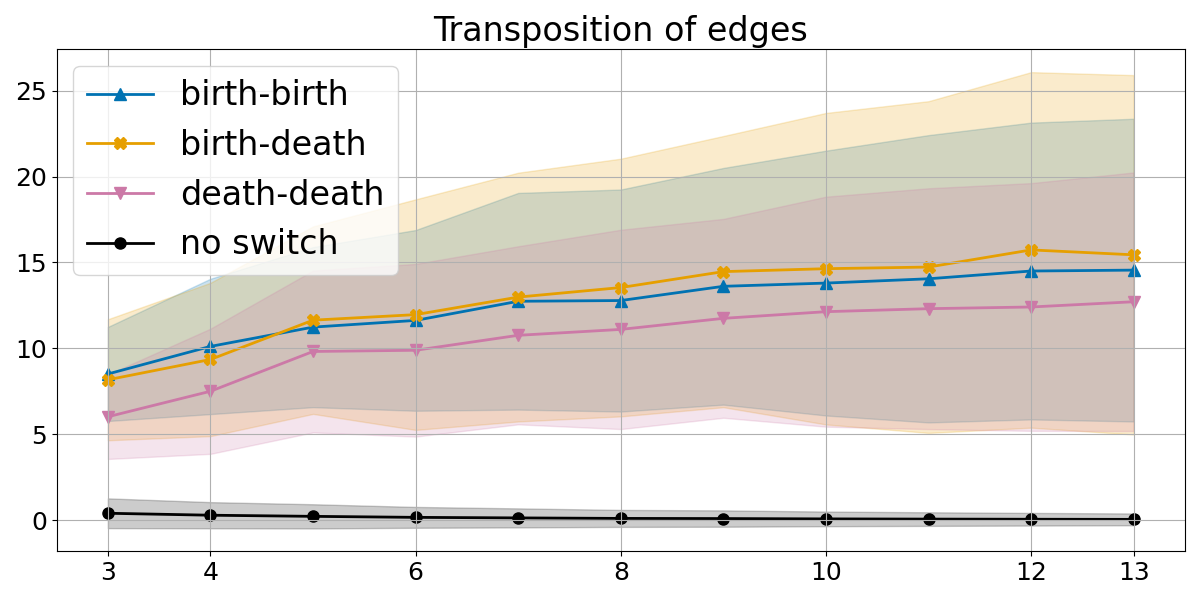}
  \hspace{0.1in} \includegraphics[width=0.28\textwidth]{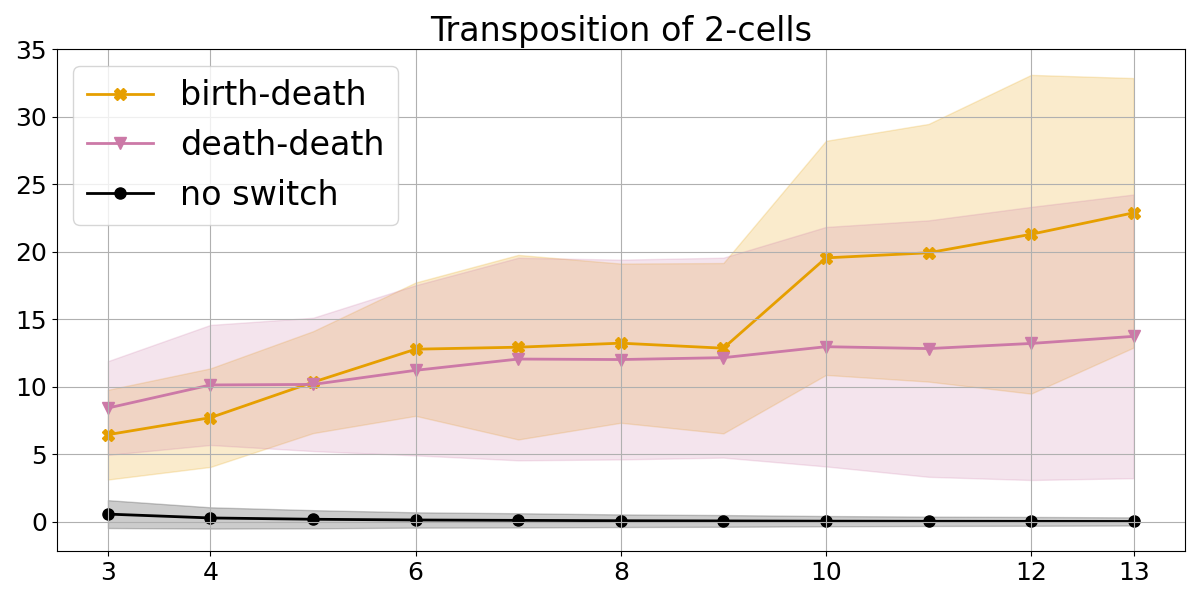} \\
  \vspace{-0.05in}
  \caption{\footnotesize \emph{From left to right}: the average number of arcs in the depth poset (\emph{upper row}) and in the transitive reduction of the depth poset (\emph{lower row}) that are different before and after the transposition of two vertices, two edges, and two $2$-cells during a straight-line homotopy between two random functions on the $2$-torus, respectively.
  We distinguish between different kinds of switches and non-switches.}
  \label{fig:changes}
\end{figure}
As shown in the upper row of Figure~\ref{fig:transposition}, the transpositions that are not switches outnumber those that are switches, so it is useful to draw the curves of the latter on a more appropriate scale, which we do in the lower row of the same figure.
As shown in the lower row of Figure~\ref{fig:transposition}, there are barely any BD-type switches of two vertices or two $2$-cells, and even for edges, there are fewer BD-type switches than switches of BB-type and switches of DD-type.

\smallskip
An interesting aspect of the transpositions is how much they change the depth poset.
Every switch affects exactly two nodes, while a transposition that is not a switch leaves all nodes unchanged.
The effect on the arcs is less predictable, so we measure the change by counting the arcs that are different before and after the transposition.
Not surprisingly, there are barely any arcs that change when the transposition is not a switch.
Also not surprisingly, there are significantly more arcs that change in the depth poset as compared to the transitive reduction of the poset.
Because of the limited size of the random functions in our experiment, it is difficult to gauge how the size of the symmetric difference of the arcs before and after the transposition grows with the size of the function.
In any case, it seems to grow much slower than the total number of arcs, which we observe to grow quadratically in $n$; see the middle panel in the lower row of Figure~\ref{fig:nodes_arcs}.

%% \newpage
%%%%%%%%%%%%%%%%%%%%%%%%%%%%%%%%%%%%%%%%%%%%%%%
%%%%%%%%%%%%%%%%%%%%%%%%%%%%%%%%%%%%%%%%%%%%%%%
\section{Discussion}
\label{sec:5}
%%%%%%%%%%%%%%%%%%%%%%%%%%%%%%%%%%%%%%%%%%%%%%%
%%%%%%%%%%%%%%%%%%%%%%%%%%%%%%%%%%%%%%%%%%%%%%%

The main contribution of this paper is the complete analysis of how transpositions of cells in the filter affect the depth poset of a filtered Lefschetz complex.
This analysis provides further evidence for the relative significance of birth-death switches, as they mark a more extensive re-organization of the persistent homology than all other types of transpositions.
Our initial computational experiments suggest that the birth-death switches are more rare than all other types, and that the number of arcs that change grows much slower than the total number of arcs in the depth poset.
This work calls for further research in at least the following two directions:
\smallskip \begin{itemize}
  \item Extend the notion of a depth poset from discrete Morse functions to discrete vector fields as introduced in Forman~\cite{For98b}, or even to multi-vector fields as studied in \cite{LKMW23,Mro17}.
  \item Study the depth poset stochastically for various models of random functions and for homotopies between them.
\end{itemize} \smallskip
We finally mention that the algorithms in this paper and its precursor \cite{ELMS25} are closely related to the algorithms for topology optimization described by Nigmetov and Morozov~\cite{NiMo24}.
Can the more global topological information provided by the depth poset be harvested to get faster or better methods for topology optimization?

% \subsection*{Acknowledgment}
% The authors thank Jakub Le\'skiewicz and Bartosz Furmanek for discussions that helped improve the paper.

\newpage
%%%%%%%%%%%%%%%%%%%%%%%%%%%

\end{document}